\documentclass[12pt]{amsart}
 \pdfoutput=1 
\usepackage{marginnote}
\usepackage{geometry}
\usepackage{cancel}
\usepackage{amsmath,amsthm, yhmath}
\allowdisplaybreaks[1]
\usepackage{amssymb,mathrsfs}

\usepackage{array}   
\newcolumntype{L}{>{$}l<{$}} 

\usepackage{tensor}

\usepackage{slashed}

\usepackage{wrapfig}
\usepackage[pdftex]{graphicx}
\usepackage{pdfpages}

 \usepackage{lmodern}
 \usepackage[T1]{fontenc}

\usepackage{esint}
\usepackage{epstopdf}
\usepackage{amscd}
\usepackage{enumerate}
\usepackage{mymathsty}
\usepackage{fancyhdr} 
\usepackage{microtype}
\usepackage[pdfencoding=auto,psdextra]{hyperref}
\usepackage[all]{xy}
\xyoption{poly}

 \advance \topmargin by -\headheight
 \advance \topmargin 
 by -\headsep
 \evensidemargin \oddsidemargin
 \numberwithin{equation}{section}

	\DeclareMathOperator{\Hess}{Hess}
	\DeclareMathOperator{\grad}{grad}

	\newcommand*{\op}[1]{{\mathrm{#1}}}
	
	\newcommand{\modop}{\slashed \bigtriangleup} 
	\newcommand{\logmodop}{\slashed \bigtriangledown}

	\DeclareMathSymbol{\Gamma}{\mathalpha}{operators}{0}
	\DeclareMathSymbol{\Delta}{\mathalpha}{operators}{1}
	\DeclareMathSymbol{\Theta}{\mathalpha}{operators}{2}
	\DeclareMathSymbol{\Lambda}{\mathalpha}{operators}{3}
	\DeclareMathSymbol{\Xi}{\mathalpha}{operators}{4}
	\DeclareMathSymbol{\Pi}{\mathalpha}{operators}{5}
	\DeclareMathSymbol{\Sigma}{\mathalpha}{operators}{6}
	\DeclareMathSymbol{\Upsilon}{\mathalpha}{operators}{7}
	\DeclareMathSymbol{\Phi}{\mathalpha}{operators}{8}
	\DeclareMathSymbol{\Psi}{\mathalpha}{operators}{9}
	\DeclareMathSymbol{\Omega}{\mathalpha}{operators}{10}

	\renewcommand{\epsilon}{\varepsilon}
	
	\DeclareMathSymbol{\T}{\mathbin}{AMSb}{"54}
	
\usepackage{mathtools}
\newcommand{\defeq}{\vcentcolon=}

\newcommand{\pFq}[2]{\ensuremath{
        {\phantom{\,}}_{#1}F_{#2}
}
    }
\newcommand*{\Scale}[2][4]{\scalebox{#1}{$#2$}}

\begin{document}
 \author{Yang Liu}
 \address{
 Max Planck Institute for Mathematics,
 Vivatsgasse 7,
 53111 Bonn,
 Germany
 }
 \email{liu.858@osu.edu}
\title{
    Hypergeometric function and modular curvature
}

 \date{\today} \thanks{} 

\begin{abstract}
We first show that hypergeometric functions appear naturally as spectral
functions when applying pseudo-differential calculus to decipher heat kernel
asymptotic in the situation where    the symbol algebra is noncommutative.
Such observation leads to a unified (works for arbitrary dimension) method of 
computing the modular  curvature on toric noncommutative manifolds.
We show that the spectral functions that define the
quantum part of the curvature  have closed forms in terms of hypergeometric
functions.  As a consequence,
we are able to obtained explicit
expressions (as functions in the dimension parameter)
     for those spectral functions without using
symbolic integration. A surprising geometric consequence is that the
functional relations coming from the variation of  the associated
Einstein-Hilbert action still hold when the dimension parameter takes real
values.

\end{abstract}

\maketitle

\tableofcontents

\section{Introduction}

%
%
%
%
\subsection{Modular geometry}
In noncommutative geometry, a notion of intrinsic curvature for certain
class of noncommutative manifolds, called modular
(scalar) curvature, has only recently begun to be comprehended.  The first
example on which the computation was  carried out in great detail is
noncommutative two tori
associated to conformal change (with respect to the flat one) of metrics, 
cf. \cite{MR3194491,MR3148618,MR3540454},  
for noncommutative four tori, \cite{MR3369894,MR3359018}. Later in
\cite{Liu:2015aa,LIU2017138}, the author
extended such  approach  to all even dimensional
toric noncommutative manifolds (also known as Connes-Landi noncommutative
manifolds, cf. \cite{MR1184061,MR2403605,MR1937657,MR1846904}). 

The word ``modular'', taken from modular theory for von Neumann algebras,
emphasizes  the new ingredients (beyond the Riemannian part
of the curvature) brought in by the noncommutativity of the metric. 
In Riemannian geometry,
the conformal class of metrics $[g]$ can be parametrized by a commutative
coordinate, more precisely, a positive smooth function called a Weyl factor $k
= e^h$. Sometimes, it is more convenient to work with its logarithm $h = \log
k$ which is a real-valued coordinate. 
For the noncommutative manifolds that have been studied, whose topological
structure is presented by a $C^*$-algebra, the Weyl factor $k$ becomes
a noncommutative coordinate, namely, a positive invertible elements
in the ambient $C^*$-algebra, which can be written as an exponential $k = e^h$
with $h$ self-adjoint. Local geometric invariants such as Riemannian curvature
are extracted from coefficients $ V_j(a,\Delta_k)$, viewed as functionals (in
$a$)  on the underlying $C^*$-algebra, of the heat kernel expansion:    
\begin{align}
    \label{eq:intro-heat-expansion}
    \Tr(a e^{-t\Delta_k} ) \backsim_{t\searrow 0}
    \sum_{j=0}^\infty V_j(a,\Delta_k) t^{(j-m)/2}, 
\end{align}
where $m$ is the dimension of the manifolds and $\Delta_k$ is a perturbation
 of the  Laplacian $\Delta\defeq \Delta_g$ associated to a background metric $g$,
which behaves like a Riemannian one and admits a pseudo-differential calculus,
cf. \cite{LIU2017138}. Intuitively, one can think of $\Delta_k$ as
a quantization \footnote{The word
``quantization'' simply means making something into an operator.} of the
conformal change of metric:
$g' = k g$.  In this paper, we only consider the simplest perturbation:
$\Delta_k \defeq k \Delta$. By analogy with the results in Riemannian geometry
(cf. \cite{MR0400315}), we define  the scalar curvature of the metric
$\Delta_k$ to
be the functional density  $R_{\Delta_k}$ of the
second term in \eqref{eq:intro-heat-expansion}:  
\begin{align}
    \label{eq:intro-V2-functional-density}
    V_2(a, \Delta_k) = \varphi_0(a R_{\Delta_k}) 
\end{align}
with respect to $\varphi_0$,  a tracial functional defined by the volume form
of the background metric $g$. The trace property of $\varphi_0$ is
a consequence of the Riemannian features of the background metric. In contrast, 
the volume form of metric $g'$ or $\Delta_k$ is associated to  
the functional $V_0(a,\Delta_k)$ in \eqref{eq:intro-heat-expansion}, which 
is a only state. The modular theory (Tomita-Takesaki
theory) asserts that, for such a state, there exists a one parameter group
$\sigma_t$ of automorphisms of the ambient von Neumann algebra measuring to
which extent the state fails to be a trace. A generator of $\sigma_t$ is called
a modular operator.  In our case, the modular operator $\mathbf y$ is  simply
the conjugation by the Weyl factor:
\begin{align}
    \mathbf y(a) = k^{-1} a k, \,\,\,
    \sigma_t(a) = \mathbf y^{it}(a), 
    \label{eq:intro-defn-modop-oneparagp}
\end{align}
where $a$ belongs to the von Neumann algebra. In terms of the logarithm
coordinate $h =\log k$, we can consider the modular derivation: $\mathbf
x \defeq \log \mathbf y = -\mathrm{ad}_h = [\cdot,h]$, which is new type of
differential generated by the modular automorphisms that one does not see in
the commutative setting.  
For higher $V_j$'s, the action of
the modular operator or modular derivation is realized by some intriguing
spectral functions such as 
$K_{\Delta_k}$  and $H_{\Delta_k}$ appeared in Eq. \eqref{eq:psecal-R_Delta_k}. 
In dimension two \cite{MR3540454,MR3194491}, the one variable function $K$ is
the Bernoulli generating function, which is related to the Todd class in
topology. In dimension four \cite{LIU2017138}, it is
the product of the exponential function and the $j$-function:
$e^x(\sinh(x/2)/(x/2))$, which shows striking similarity to the Atiyah-Singer
local index formula. 
Seeking for deeper understanding of those spectral functions and related
functional relations, both conceptually and computationally, is one of the main
motivations that are pushing the project forward.  

\subsection{Pseudo-differential calculi and hypergeometric functions}
The technical tool that being used to decipher the heat kernel expansion is
a pseudo-differential calculus which is suitable for studying the spectral
geometry of the underlying manifolds, such as Connes's calculus  for
noncommutative tori \cite{connes1980c}
and deformation of Widom's calculus for general toric noncommutative manifolds
\cite{Liu:2015aa}. 
In general, a pseudo-differential calculus provides
a recursive algorithm to construction a sequence of symbols
$\set{b_j}_{j=0}^\infty$, whose
truncated sums $\sum_0^N b_j$ approximate the resolvent of the elliptic
operator in question, as $N\rightarrow \infty$, so that the  
heat coefficients $V_j$ can be obtained  by certain integration of the
corresponding
$b_j$, see \eqref{eq:psecal-mathcalR_j}.
The first approximation   $b_0 = (p_2-\lambda)^{-1}$ is the resolvent of $p_2$,
the leading symbol 
of the elliptic operator in question. Higher $b_j$'s  are finite
sums of the form:   
\begin{align}
    \label{eq:intro-b_j-Fsum}
   b_j = \sum b_0^{l_0} \rho_1 b_0^{l_1} \rho_2 \cdots b_0^{l_{n-1}}
    \rho_n b_0^{l_n} ,
\end{align}
where the exponents $l_j$'s are non-negative integers and $\rho_j$'s are the
derivatives of symbols of the elliptic operator. 
Notice that $b_0$ and $\rho_j$'s do
not commute in general, even in the commutative case  in which the symbols are
endomorphism-valued sections acting on some vector bundle. 
A useful trick  to handle the
non-commutativity is to rewritten the summands in
\eqref{eq:intro-b_j-Fsum} as ``contractions'', cf.
\cite{leschdivideddifference}:
\begin{align}
    \label{eq:intro-contractions}
     ( b_0^{l_0} \otimes \cdots \otimes b_0^{l_n})
    \cdot (\rho_1 \otimes \cdots \otimes \rho_n)
    \defeq b_0^{l_0} \rho_1 b_0^{l_1} \rho_2 \cdots b_0^{l_{n-1}}
    \rho_n b_0^{l_n},
\end{align}
because eventually, $\rho_j$'s can be factored out of the integration. 
The contribution to the action of the modular operator is given by the
operator-valued integral:
\begin{align*}
    &\,\, \int_0^\infty \frac{1}{2\pi i} \int_C
    e^{-\lambda} b_0^{l_0} \otimes \cdots \otimes b_0^{l_n} d\lambda 
    (r^{m-1} dr )\\
    =&\,\, 
\int_0^\infty \frac{1}{2\pi i} \int_C
e^{-\lambda} (k r^2-\lambda)^{-l_0}\otimes \cdots
\otimes (k r^2-\lambda)^{-l_n} d\lambda 
(r^{m-1} dr ),
\end{align*}
where $m$ is the dimension of the manifold and $k$ is 
a positive invertible element in some $C^*$-algebra. Method used in the
previous works  involves switching the order of integration, which depends on $m$
is a subtle way. In this paper, we attack  the contour integral directly by
replacing each $(kr^2-\lambda)^{-l_j}$ with its Mellin transformation
\eqref{eq:hgeofun-Mellin-Tran.}. A surprising  consequence is that the integral
turns out to be (upto a constant factor) a hypergeometric function or its
multi-variable generalization. We  shall mainly focus on 
 the case in which $n$
(appeared in  \eqref{eq:intro-contractions}) equals $1$ or $2$, since it is
sufficient for the studying of the $V_2$-term. 

By landing the building blocks of the spectral functions into the hypergeometric
family, we obtain tons of new functional relations,  among
which, the differential and contiguous relations are the fundamental ones.
Certain patterns of differential relations have already been observed in
\cite{leschdivideddifference,2016arXiv161109815C}.  The contiguous relations
seems like the other side of a coin to the differential relations, which lead
us to some number theoretical properties, such as  Gauss's continued
fraction, see Eq. \eqref{eq:hggeofun-DR-tildeKan-GCF}. From the computational
side, symbolic integration is replaced by differentiation and recurrence
relations. In particular,  we have remove the restriction in  \cite{LIU2017138}
that the manifold has to be even dimensional; we have recurrence relations
among all even dimensions and all odd dimensions; we are able to express the
spectral functions explicitly in the dimension parameter $m$.

\subsection{Variation of the EH (Einstein-Hilbert) action}
The geometric application of the local expression of the modular scalar
curvature studied in this paper is the variation of the 
 EH-functional/action. In our operator theoretical framework, it
is encoded in the  second heat coefficient
$V_2(a,\Delta_k)$, which can be viewed as a function in both $a$ and $k$. 
When the metric coordinate $k$ is fixed, we vary $a$ to recover the functional density,
which is the modular scalar curvature \eqref{eq:intro-V2-functional-density}.
On the other hand, we can take $a = 1$ and view it a functional in $k$, 
\begin{align*}
    k \mapsto V_2(1,\Delta_k) = \varphi_0 (R_{\Delta_k}),
\end{align*}
where the right hand side  mimics the integration of the scalar curvature
against the corresponding volume form $\varphi$ in Eq.
\eqref{eq:varprob-volume-varphi}\footnote{
    Notice that in Eq. \eqref{eq:psecal-R_Delta_k}, the volume factor $k^{j_m}$
    is already included in $R_{\Delta_k}$,
    therefore, one should apply the functional $\varphi_0$, instead of
    $\varphi$, to recover the EH-action.
}, which
is the EH-action in the commutative setting. The modular
curvature $R_{\Delta_k}$ itself involves two spectral functions $K_{\Delta_k}$
and $H_{\Delta_k}$ of one and two variables respectively. On the other hand,
its integration 
$\varphi_0(R_{\Delta_k})$, can be described by  only  one
spectral function $T_{\Delta_k}$, which is determined, of course, by $K_{\Delta_k}$
and $H_{\Delta_k}$, see \eqref{eq:varprob-defn-TDleta}. It was shown in
\cite{LIU2017138}  that the
EH-action determines the scalar curvature functional completely, namely, if we know
$V_2(1,\Delta_k)$ for all positive invertible elements $k$, then we can recover
$V_2(a,\Delta_k)$ for any ``function'' $a$. The proof is carried out by computing
the functional gradient (by means of G\^{a}teaux differential, see Definition
\ref{defn:varprob-gradFEH})  in two
different ways: by differentiating the local expression $\varphi_0
(R_{\Delta_k})$ and the trace of the heat operator $\Tr (e^{-t\Delta_k})$
respectively. The later one is easier to compute, which shows that the gradient
is the almost the same as the scalar curvature, in particular, the spectral
functions are still $K_{\Delta_k}$ and $H_{\Delta_k}$. This is
a noncommutative
analog of the following fact in Riemannian geometry: for the conformal change
of metric 
$g' = u^{4/(m-2)}g$, one sees the scalar curvature of $g'$ in  the
Euler-Lagrange equation of the Yamabe functional, which is a normalization of
the EH-functional, cf. \cite[Eq. 1.11]{yamabe1960}. For the former one, the
calculation, which is discussed in Section
\ref{sec:varcalwrptncvar}, is much longer. The bottom line is that the spectral
functions of
the gradient can be completely derived from $T_{\Delta_k}$,  see Eq.
\eqref{eq:varprob-IT-explicit}, by two basic
transformations.  One comes from integration by parts:
$T_{\Delta_k}(u) \mapsto
T_{\Delta_k}(u^{-1})$, see \eqref{eq:varprob-integrationbyparts-I}. The other
arises from commutators of 
 the covariant differential  (commutative differential) and  the modular
operator/derivation (noncommutative differential):
\begin{align*}
T_{\Delta_k}(u) \mapsto D( T_{\Delta_k})(u,v) = \frac{T_{\Delta_k}(uv)
- T_{\Delta_k}(u) }{
T_{\Delta_k}(v) -1 }.
\end{align*}
Such phenomenon  fits with the classical picture that the curvature $R(X,Y)$ is
designed to
measure the commutator $[\nabla_X ,\nabla_Y]$ of different differentials. By
equating two approaches, we
obtain functional relations
stated in Theorem \ref{thm:varprob-thm-FunRelation-T-to-K-Ha} which tell us 
exactly how to reproduce $K_{\Delta_k}$ and $H_{\Delta_k}$ from $T_{\Delta_k}$.

A new input in present paper is that the variation is performed with respect to the
Weyl factor $k$ itself, instead of $h=\log k$ as in all previous works
\cite{MR3540454,MR3194491,MR3369894,LIU2017138}. One of the advantages is to
get rid of the spectral functions, cf. \cite[Eq. (6.9), (6.10)]{MR3194491},
when computing $\nabla e^h$ and $\nabla^2 e^h$. Such modification is merely an
operation like changing coordinates, which alters nothing of the underlying
geometric objects. Nevertheless, picking a good coordinate system could be very
helpful when attacking specific problems. 


  The discussion of EH-action above requires that the
dimension of the manifold $m$ is greater than $2$. In fact, in
dimension two,
 the celebrated Gauss-Bonnet theorem  asserts that the EH-action is
a constant function whose value, upto a normalization, equals the Euler
characteristic of the given manifold. The corresponding result for
noncommutative two tori was studied in \cite{MR2907006,MR2956317}. One can also
see this from the factor $(2-m)/2$ appeared in the right hand side of Eq.
\eqref{eq:varprob-V_j-step1}. An interesting functional with non-trivial 
variation in dimension two is
the Ray-Singer log-determinant functional, which has been intensively studied in  
\cite{MR3194491}.

The last  contribution of this paper is the final expressions of $K_{\Delta_K}$,
$H_{\Delta_k}$ and $T_{\Delta_k}$ (see Eq.
\eqref{eq:checkrels-explict-KDelta}, \eqref{eq:checkrels-explict-HDelta} and
\eqref{eq:checkrels-explict-TDelta}) as functions of $m$, the
dimension parameter, thanks to the available knowledge of hypergeometric functions.
One can performed straightforward verification with the assistant of
\textsf{Mathematica} to see that the three
functions satisfy the equations given in Theorem
\ref{thm:varprob-thm-FunRelation-T-to-K-Ha} for all $m \in (2,\infty)$.
For $m=2,3,\dots$ being integers, such verification gives conceptual
confirmation for the validity of the lengthy calculation. While for real values
$m$, functional relations  is a surprising but exciting by product, which do
not admit a geometric proof so far.  Nevertheless,
the geometric significance is that it extends the ``universality'' of those
spectral functions in the sense that  the pseudo-differential approach of
studying local geometry has the
potential to be applied onto  noncommutative manifolds of non-integer dimensions. In
our operator theoretical framework, the (metric) dimension  is determined by
the Weyl's law, that is, by
the rate of growth of the spectrum of the Laplacian
operator, thus it takes real values in a natural way.  
We refer to  \cite[Sect. 10.2]{connes2008noncommutative} for detailed
explanation and to   \cite{MR1303779} for examples.  

\subsection{Outline of the paper}
In section \ref{sec:spfun-hygeofun}, we show that the Euler type integral
representations of  hypergeometric functions $\pFq21$ and Appell's $F_1$
functions appears naturally in the pseudo-differential approach of heat kernel
expansion. Results of this type are called rearrangement lemmas in the literature 
\cite{leschdivideddifference,MR3194491,MR2907006}. The most general form so far
is given 
in Prop. \ref{prop:hgfun-n-var}.  
Such observation brings in powerful tools for computations. For
example, we have elegant 
differential and contiguous functionals relations to replace symbolic
integrations. Among which,  
Prop. \ref{prop:hgeofun-cor-F1-dividedidff} is the most important technical
result, which reduces double integrals (Appell's $F_1$ functions) to a divided
difference of hypergeometric functions. Section \ref{sec:varcalwrptncvar}
and \ref{sec:var-FEH} contains the geometric part:  the variation of the EH-action.
We try our best to
explore the connections between ``individual'' functionals relations among
hypergeometric  
functions and the ``global'' relations stated in Theorem
\ref{thm:varprob-thm-FunRelation-T-to-K-Ha}. For completeness, we outline the
calculation for the modular curvature  on noncommutative tori of arbitrary
dimension $m\ge 2$ in the last section. 

\section{Hypergeometric functions in heat kernel expansion}
\label{sec:spfun-hygeofun}
In this section, we give a new method to prove and generalize the rearrangement
lemmas in the literature \cite{MR2907006,MR3194491,leschdivideddifference}.
A surprising discovery is the appearance of hypergeometric function and its
multivariable generalizations in the part of the spectral functions. The most
general version so far is given in Prop. \ref{prop:hgfun-n-var}.


\subsection{Contour integral for the heat operator}
We start with a lemma which handles the contour integral that defines the heat
operator. In our applications,  the elliptic operator $P$ has spectrum
contained in $[0,\infty)$. Therefore, we can choose
the   contour $C$ to be   the imaginary axis from $-i\infty $ to $i\infty$, so
that 
\begin{align*}
    e^{-tP} = \frac{1 }{2\pi i} \int_C e^{-t\lambda} (P-\lambda)^{-1} d\lambda
    = (2\pi)^{-1} \int_0^\infty e^{-tix} (P-ix)^{-1} dx.
\end{align*}

\begin{lem}
    \label{lem:hgeofun-contourint}
    Let $A,B,C$ be  positive real numbers and $a,b,c$ be  negative integers, 
    we have
    \begin{align}
        \label{eq:hgeofun-lem-cintAB}
        (2\pi)^{-1} \int_{-\infty}^{\infty} e^{-ix} 
    (A - ix)^a ( B - ix)^b dx  =
    \frac{1 }{\Gamma(a) \Gamma(b) }\int_0^1  (1-t)^{a-1} t^{b-1}  
e^{-(A - (A-B)t)} dt
    \end{align}
    and
    \begin{align}
\label{eq:hgeofun-lem-cintABC}
\begin{split}
    &\,\,\,   (2\pi)^{-1} \int_{-\infty}^{\infty} e^{-ix} 
         (A - ix)^a ( B - ix)^b (C-ix)^c dx  \\
       =  &\,\,\,  \frac{1 }{\Gamma(a) \Gamma(b) \Gamma(c)}
         \int_0^1 \int_0^{1-t} (1 - t -u)^{a-1} t^{b-1} u^{c-1}   
    e^{
    -(A -(A -B)t - (A - C)u )
} dudt.
\end{split}
       \end{align}
\end{lem}
\begin{rem}
   Notice that the  right hand side of \eqref{eq:hgeofun-lem-cintAB} is a confluent
    hypergeometric function:
    \begin{align}
        (2\pi)^{-1} \int_{-\infty}^{\infty} e^{-ix} 
        (A - ix)^a ( B - ix)^b dx  =  \frac{e^{-A}}{ \Gamma(a+b)} 
        \pFq11(a;a+b;B-A),
        \label{eq:hgeofun-F11-heat-contourint}
    \end{align}
    where the function $ \pFq11(a;b;B-A)$ has the following integral
    representation:
    \begin{align}
        \pFq11(a;b;z)= \frac{
            \Gamma(a-b) \Gamma(b)
        }{ \Gamma(a)}\int_0^1 t^{a-1} (1-t)^{b-a-1} e^{zt} dt.
        \label{eq:hgeofun-F11-int-rep-defn}
    \end{align}
     The identity  \eqref{eq:hgeofun-F11-heat-contourint} appeared   in the
     heat kernel related work \cite{MR1156063} and \cite{MR1843855}. This was the
     motivation at the early stage that brought the author's attenstion to 
     hypergeometric functions. 
 \end{rem}

\begin{proof}
    We only prove \eqref{eq:hgeofun-lem-cintABC} and leave
    \eqref{eq:hgeofun-lem-cintAB} to the reader. Observe  that Powers like
    $(A-ix)^a$ can be rewritten in terms of Mellin transform: 
    \begin{align}
        \label{eq:hgeofun-Mellin-Tran.}
        (A-ix)^a = \frac{1}{\Gamma(a)}\int_0^\infty s^{a-1} e^{-s(A-ix)} ds.
    \end{align}
    The rest of the computation is straightforward, denote  $\gamma(a,b,c) = 
    [\Gamma(a) \Gamma(b) \Gamma(c)]^{-1} $, then:
    \begin{align*}
        &\,\, (2\pi)^{-1} \int_{-\infty}^{\infty} e^{-ix} 
    (A - ix)^a ( B - ix)^b (C - ix)^c dx \\ =& \,\, 
    (2\pi )^{-1} \gamma(a,b,c) \int_{-\infty}^{\infty}
  \int_{[0,\infty)^3} 
      e^{-ix}
      s^{a-1}  t^{b-1}  u^{c-1} 
        e^{-s(A-ix) -t(B- ix) - u(C - ix)} dsdtdu  dx \\
=&\,\, 
 \gamma(a,b,c) \int_{[0,\infty)^3}  s^{a-1}  t^{b-1}  u^{c-1} 
      e^{-sA -Bt-Cu} dsdtdu
      \brac{ (2\pi)^{-1} \int_{-\infty}^{\infty}
       e^{ix(s+u+t-1)} dx } \\
    =& \,\,
   \gamma(a,b,c) \int_0^1 \int_0^{1-t} (1 - t -u)^{a-1} t^{b-1} u^{c-1}   
    e^{
    -(A -(A -B)t - (A - C)u )
    } dudt.
    \end{align*}
    For the last equal sign, we use the fact that  $(2\pi)^{-1} \int_{-\infty}^{\infty}
    e^{ixy} dx$ equals the Dirac-delta distribution $\delta(y)$, therefore the
    domain of  integration  is reduced  from $[0,\infty)^3$ to $$\set{s+t+u=1,
    \,\, s,t,u>0}.$$ 
\end{proof}

\subsection{Spectral functions in terms of hypergeometric functions}
\label{sec:spec-using-hgfuns}

\label{sec:notations-from-Les17}


Let $\mathcal A$ be a unital $C^*$-algebra. For elementary tensors $a = (a_0 \dots,
a_n) \in \mathcal A^{\otimes n+1}$ and $\rho = (\rho_1, \dots , \rho_n) \in
\mathcal A^{\otimes n}$, we recall the contraction $a . \rho$ used in \cite{leschdivideddifference}: 
\begin{align*}
    a. \rho := a_0 . \rho_1 . a_1. \rho_2 \cdots a_{n-1} \rho_n . a_n.
\end{align*}

For a single element $k \in \mathcal A$ acting on a elementary tensor $\rho$ above (or a product $\rho_1 \cdots \rho_n$), we introduce multiplication in the $j$-th slot:
\begin{align}
    k^{(j)} \defeq (1,\dots, k,1\dots,1),\,\,\, \text{$k$ is in the $j$-th slot, 
    $j=0,1,\dots,n$,}
    \label{eq:hgeofun-k^j}
\end{align}
and conjugation on the $j$-th factor:
\begin{align}
    \mathbf y_j(k)\defeq (1,\dots,k^{-1}, k,1\dots,1), \,\,\,
    \text{$k^{-1}$ is in the $( j-1 )$-th slot, 
    $j=0,1,\dots,n$.}
    \label{eq:hgeofun-y_j(k)}
\end{align}
Later, we need to replace all the multiplication $k^{(j)}$, $j\ge 1$ by conjugation operators: 
\begin{align}
    \label{eq:hgeofun-kjandy_j}
    k^{(j)}  = k^{(0)} \mathbf y_1(k) \cdots \mathbf y_j(k).  
\end{align}
In the rest of the paper, one can take $\mathcal A$ to be a noncommutative
$m$-torus: $C^\infty(\T^m_\theta)$ (cf. \cite{MR1047281}).

When applying pseudo-differential calculus to compute heat kernel asymptotic,
one encounter integrals involving the resolvent approximation symbols $b_j$:
\begin{align*}
    \int_0^\infty \frac{1}{2\pi i} \int_C e^{-\lambda} b_j(r,\lambda) 
    d\lambda r^{(m-1)/2} dr.
\end{align*}
In general, $b_j$ is a finite sum of the form:
\begin{align}
    \label{eq:hgeofun-summandsb_j}
    \sum  b_0^{j_0} . \rho_1 .  b_0^{j_1} . \rho_2 \dots \rho_l
    . b_0^{j_l},
\end{align}
where $b_0 = (k r^2 - \lambda)^{-1}$ and $\rho_{p}$  are of
the form $k^{a_p} \nabla^{b_p} k$ for some integers $a_p$ and $b_p$, where
$p=0,\dots,l$. 

%

Following the contraction notation in the previous section:
\begin{align*}
    & \,\, \int_0^\infty \frac{1}{2\pi i} \int_C e^{-\lambda} b_j(r,\lambda) 
    d\lambda r^{(m-1)/2} dr \\
    = &\,\,  \brac{
\int_0^\infty \frac{1}{2\pi i} \int_C e^{-\lambda}
(k^{(0)} r^2 - \lambda)^{-j_1}\cdots (k^{(n)} r^2 - \lambda)^{-j_l}
d\lambda r^{(m-1)/2} dr }
 \cdot (\rho_1 \cdots \rho_n).
\end{align*}

In this paper, we only need to integrate $b_2$ which involves terms 
with $l=1$ or $2$.

\begin{prop}
    \label{prop:hgeofun-one-var-family}
    Let  $a,b,m$ be positive integers, we abbreviate the conjugation operator
    defined in \eqref{eq:hgeofun-y_j(k)} as: $\mathbf y(\rho):= \mathbf y_1(k)(\rho)
    = k^{-1} \rho k$, then    
\begin{align}
    \begin{split}
    &\,\, \int_0^\infty \frac{1}{2\pi i} \int_C e^{-\lambda}
    (k^{(0)} r^2 - \lambda)^{-a} \cdot \rho \cdot  (k^{(1)} r^2 - \lambda)^{-b}
 d\lambda (r^{2 d_m +1} dr)   
  \\ = & \,\, 
      (k^{(0)})^{-(d_m+1)}
  \brac{
      \frac{\Gamma(  d_m  +1)}{2 \Gamma(a+b)} 
      \pFq21 (d_m+1;b,b+a;1- \mathbf y)      
  } \cdot \rho.
    \end{split}
 \label{eq:hgeofun-form-one-var}
\end{align}
where $d_m:= d_m(a,b)= a+b -2 + (m-2)/2$, and the contour $C$ can be taken
to be the imaginary axis  from $-i\infty$ to $i\infty$. 
We denote the result as
\begin{align}
    K_{a,c}(y;m) =  
      \frac{\Gamma(  d_m  +1)}{ \Gamma(a+b)} 
      \pFq21 (d_m+1;b,b+a;1-  y)      
    \label{eq:hgeofun-onevar-K_acm}
\end{align}
for later discussion. 
\end{prop}
\begin{rem}
    The appearance of $d_m(a,b)$ is due to the homogeneity, degree $-4$, of $b_2$ in
    $r$ (before the substitution $r \mapsto r^2$): 
    $b_2(c r) = c^{-4} b_2(r)$. In general, $b_j$ is of homogeneity $-2-j$ when
    the corresponding differential operator is of order $2$. 
\end{rem}
Similarly, we have a two-variable version which gives rise to the Appell
hypergeometric functions. 
\begin{prop}
    \label{prop:hgeofun-two-var-family}
    Let  $a,b,c,m$ be positive integers, we abbreviate the conjugation operator
    defined in \eqref{eq:hgeofun-y_j(k)} as: $ \mathbf y_j :=\mathbf y_j(k)$,
    with $j=1,2$. 
\begin{align}
    \begin{split}
  & \,\, \int_0^\infty 
  \frac{1}{2\pi i} \int_C e^{-\lambda}
  (k^{(0)} r^2 - \lambda)^{-a} \cdot \rho_1 \cdot  (k^{(1)} r^2 - \lambda)^{-b}
\cdot \rho_2 \cdot
  (k^{(2)} r^2 - \lambda)^{-c}
  d\lambda (r^{2d_m +1} dr)
 \\ &= \,\,
      (k^{(0)})^{-(d_m+1)}
   \brac{
      \frac{\Gamma(  d_m  +1)}{2 \Gamma(a+b +c)} 
      F_1(d_m+1;c,b;a+b+c;1-\mathbf y_1\mathbf y_2, 1-\mathbf y_1 )      
   } \cdot (\rho_1 \rho_2).   
    \end{split}
   \label{eq:hgeofun-form-two-var}
\end{align}
where $d_m:= d_m(a,b,c)= a+b+c -2 + (m-2)/2$, and the contour $C$ can be taken
to be the imaginary axis  from $-i\infty$ to $i\infty$. 
We denote the result as 
\begin{align}
    H_{a,b,c}(y_1, y_2;m )= 
      \frac{\Gamma(  d_m  +1)}{ \Gamma(a+b +c)} 
      F_1(d_m+1;c,b,a+b+c; 1-  y_1  y_2, 1-  y_1)      
    \label{eq:hgeofun-twovar-H_abcm}
\end{align}
for later discussion.
\end{prop}
\begin{proof}
    We first perform a substation $r \mapsto r^2$, thus the volume form
    $r^{2 d_m+1}dr$ becomes $r^{d_m}dr/2$, in which
    the overall factor $1/2$ will be omitted in the rest of the proof.
    The contour integral  can be computed
    according to    
    lemma \ref{lem:hgeofun-contourint}, in which 
     $A = k^{(0)} r$,  $B = k^{(0)} r$ and $C = k^{(2)} r$ are bounded
     operators with positive spectrum.  A 
     Fubini type theorem cf. \cite[Lemma 2.1]{leschdivideddifference} has been
     applied to makes sense of  such functional calculus. So far, we have
     \begin{align*}
   &\,\,      \frac{1}{2\pi i} \int_C e^{-\lambda}
  (k^{(0)} r^2 - \lambda)^{-a}   (k^{(1)} r^2 - \lambda)^{-b}
  (k^{(2)} r^2 - \lambda)^{-c}
 d\lambda \\
 =&\,\, \gamma(a,b,c)   \int_0^1 \int_0^{1-t} (1 - t -u)^{a-1}
 t^{b-1} u^{c-1} 
 e^{
 -  r[ k^{(0)} - t( k^{(0)} - k^{(1)}) - u ( k^{(0)} - k^{(1)})]
 }
 dudt \\ =
 &\,\,  \gamma(a,b,c)   \int_0^1 \int_0^{1-t} (1 - t -u)^{a-1}
 t^{b-1} u^{c-1} e^{
     -r[1-t(1- \mathbf y_1) -u(1- \mathbf y_1 \mathbf y_2)] 
 }
 dudt
     \end{align*}
     Next, we apply $\int_0^\infty(*) r^{d_m}dr$ to the result above and
     integrate in $r$ first:
     \begin{align*}
         & \,\,   \int_0^\infty r^{d_m} e^{
 -  r[ k^{(0)} - t( k^{(0)} - k^{(1)}) - u ( k^{(0)} - k^{(1)})]
 } dr \\
= &\,\,  (k^{(0)})^{-(d_m+1)}  
 \int_0^\infty  r^{d_m}e^{
     -r[1-t(1- \mathbf y_1) -u(1- \mathbf y_1 \mathbf y_2)] 
 } dr \\
 =&\,\, 
 (k^{(0)})^{-(d_m+1)} 
 \brac{1-t(1- \mathbf y_1) -u(1- \mathbf y_1 \mathbf y_2)}^{-(d_m+1)}
 \brac{\int_0^\infty r^{d_m} e^{-r} dr}
 \\
 =&\,\, 
 \Gamma(d_m +1)
 (k^{(0)})^{-(d_m+1)} 
 \brac{1-t(1- \mathbf y_1) -u(1- \mathbf y_1 \mathbf y_2)}^{-(d_m+1)}.
\end{align*}
An operator substitution lemma \cite[Thm
2.2]{leschdivideddifference} 
has been used  twice.
 Roughly speaking, the lemma allows us to treat
 mutually commuting positive operators $k^{(j)}$ with $j=1,2,3$, and $1-t(1-
\mathbf y_1) -u(1- \mathbf y_1 \mathbf y_2)$ as positive real numbers during
the computation.  
To get the first equal sign, we set 
$r \mapsto r k^{(0)}$, and use the relations (see Eq. \eqref{eq:hgeofun-kjandy_j}): 
\begin{align*}
    k^{(1)} = k^{(0)} \mathbf y_1, \,\,\,
    k^{(2)} = k^{(0)} \mathbf y_1 \mathbf y_2.
\end{align*}
For the second equal sign, we let $r \mapsto  r[1-t(1- \mathbf y_1) -u(1-
\mathbf y_1 \mathbf y_2)]$, which is a positive operator for $r>0$ and $0< t,u
<1$.

After combining the previous two steps, we can finish the proof by observing that  the
integral in $t$ and $u$ matches the integral representation of Appell's
hypergeometric functions, see  Eq.  \eqref{eq:hygeofun-APF1-defn}  
\begin{align*}
    &\,\,  \int_0^\infty
 \frac{1}{2\pi i} \int_C e^{-\lambda}
  (k^{(0)} r^2 - \lambda)^{-a}   (k^{(1)} r^2 - \lambda)^{-b}
  (k^{(2)} r^2 - \lambda)^{-c}
 d\lambda    
 (r^{d_m}dr) \\ 
 = & \,\,
\Gamma(d_m +1)
 (k^{(0)})^{-(d_m+1)} \gamma(a,b,c)\\
 &  \int_0^1 \int_0^{1-t} (1 - t -u)^{a-1}
 t^{b-1} u^{c-1} 
 \brac{1-t(1- \mathbf y_1) -u(1- \mathbf y_1 \mathbf y_2)}^{-(d_m+1)}
 du dt \\ =
 & \,\, 
 (k^{(0)})^{-(d_m+1)}
 \frac{\Gamma(d_m+1)}{\Gamma(a + b + c)}
 F_1(d_m+1;c,b;a+b+c; 1-\mathbf y_1\mathbf y_2, 1-\mathbf y_1 )      
\end{align*}
\end{proof}
  
\subsection{Multivariable generalization}
In general, the role of the Appell's $F_1$ in Proposition
\ref{prop:hgeofun-two-var-family} becomes Lauricella Functions of type
$F^{(n)}_D$.  We
start with the series version:
\begin{align}
    \label{eq:hgeofun-LauricellaSeries}
     F^{(n)}_D (a;\alpha_1,\dots,\alpha_n;c;x_1,\dots,x_n) 
    =  \sum_{\beta_1,\dots,\beta_n\ge 0} \frac{
        (a)_{\beta_1+\cdots\beta_n} (\alpha_1)_{\beta_1}
        (\alpha_n)_{\beta_n} x_1^{\beta_1}\cdots x_n^{\beta_n}
    }{ 
(c)_{\beta_1+\cdots\beta_n}
     \beta_1!\cdots \beta_n!
 },
\end{align}
which has the following integral representation (cf. \cite{MR0346178}):
\begin{align}
    \label{eq:hgeofun-LauricellaFuntions}
    &\,\, \int_{R_{(u_1,\dots,u_n)}} (1-\sum_{j=1}^n
    u_j)^{c-\alpha_1-\cdots-\alpha_n-1}
        u_1^{\alpha_1-1}\cdots u_n^{\alpha_n-1} 
        \cdot (1-\sum_{j=1}^n x_j u_j)^{-a} du_1\cdots du_n
       \\ =&\,\,
        \frac{
            \Gamma(\alpha_1)\cdots
            \Gamma(\alpha_n)\Gamma(c-\alpha_1-\cdots\alpha_n)
    }{\Gamma(c)}
    F^{(n)}_D (a;\alpha_1,\dots,\alpha_n;c;x_1,\dots,x_n). 
        \nonumber
\end{align}
The three variable case was first introduced by Lauricella in
\cite{Lauricella1893} and  more
fully by Appell and Kamp\'e de F\'eriet \cite{appell1926fonctions}.

 We will only repeat the major steps in the previous calculation and left
 details to
 interested readers.
Let $(A_0, \dots, A_n) \in (0,\infty)^{n+1}$. Define 
\begin{align}
    \label{eq:hggeo-Ga0toan}
 G_{\alpha_0,\dots,\alpha_n} (A_0,\dots,A_n) 
 = (2\pi)^{-1}
\int_{-\infty}^\infty e^{-i x}
(A_0 -ix)^{-\alpha_0} \cdots (A_n -ix)^{-\alpha_n} dx.
\end{align}
Similar to Lemma \ref{lem:hgeofun-contourint}, we have
\begin{align}
    \label{eq:hggeo-G-asinlemma**}
    &\,\,   G_{\alpha_0,\dots,\alpha_n} (A_0,\dots,A_n)  \\
    =&\,\, \gamma(\alpha_0,\dots,\alpha_n) \int_{R_{(u_1,\dots,u_n)}}
    (1-\sum_{j=1}^n u_j)^{\alpha_0-1} u_1^{\alpha_1-1}\dots u_n^{\alpha_n-1}
    \nonumber \\ &
    \cdot \exp\brac{
        -A_0(1-\sum_{j=1}^n u_j) - \sum_{j=1}^0 A_j u_j
    } du_1\cdots du_n,
    \nonumber
\end{align}
where $\gamma(\alpha_0,\dots,a_n)
= (\Gamma(\alpha_0)\cdots\Gamma(\alpha_n))^{-1}$ and the domain of integration 
a standard  $n$-simplex:
\begin{align*}
    R_{(u_1,\dots,u_n)} = \set{
        u_1\ge 0, \dots, u_n\ge 0, 1-\sum_{j=1}^n u_j\ge 0
    }.
\end{align*}
Next, we  set $A_j$ to be positive operators: $A_j = k^{(j)} r$ for
$j=0,1,\dots,n$ and $r\ge 0$, where $k^{(j)}$ are the Weyl factors in
\eqref{eq:hgeofun-k^j}. Again, by the operator substitution lemma (\cite[Lemma
2.1]{leschdivideddifference}): 
\begin{align*}
  &\,\, \int_0^\infty \exp\brac{
      - rk^{(0)} (1-\sum_{j=1}^n u_j) - \sum_{j=1}^0 r k^{(j)} u_j
    }  (r^d dr)  \\
    = &\,\,
    \Gamma(d+1) (k^{(0)})^{-(d+1)}  \brac{
        1 - \sum_{j=1}^m \mathbf z_j u_j  
    }^{-(d+1)}  ,
\end{align*}
where $\mathbf z_j = 1-\mathbf y_1 \cdots \mathbf y_j$. Notice that we have
used the relations \eqref{eq:hgeofun-kjandy_j}. The final goal is the following
integral:
\begin{align*}
    &\,\, \int_0^\infty G_{\alpha_0,\dots,\alpha_n}(k^{(0)}r,\dots,k^{(n)}r) (r^d dr)
    \\ =&\,\,
    (k^{(0)})^{-(d+1)} \frac{\Gamma(d+1)}{\gamma(a_0,\dots,\alpha_n)}
        \int_{R_{(u_1,\dots,u_n)}} (1-\sum_{j=1}^n u_j)^{\alpha_0-1}
        u_1^{\alpha_1-1}\cdots u_n^{\alpha_n-1} \\
        &\cdot (1-\sum_{j=1}^n \mathbf z_j u_j)^{-(d+1)} du_1\cdots du_n. 
\end{align*}
The integral on the right hand side above is a $n$-variable Lauricella function
$F^{(n)}_D$ in \eqref{eq:hgeofun-LauricellaFuntions}.
Sum up, we
have obtained the following $n$-variable version for the spectral functions.
\begin{prop}
    \label{prop:hgfun-n-var}
    Keep notations.
    \begin{align}
    \label{eq:hgfun-n-var}
    &\,\, (k^{(0)})^{d+1}\int_0^\infty
    G_{\alpha_0,\dots,\alpha_n}(k^{(0)}r,\dots,k^{(n)}r) (r^d dr)
        \\ =&\,\,
        \frac{\Gamma(d+1)}{\Gamma(\alpha_0+\dots+\alpha_n)}  
        F^{(n)}_D(d+1;\alpha_1,\dots,\alpha_n;\alpha_0+\cdots+\alpha_n;\mathbf
        z_1,\dots,\mathbf z_n),
        \nonumber
    \end{align}
    where $\mathbf z_j = 1- \mathbf y_1\cdots\mathbf y_j$ for $j=1,2,\dots,n$.
\end{prop}

\subsection{Differential and contiguous functional relations}
Let us slightly change the notations in Proposition
\ref{prop:hgeofun-two-var-family}. Denote
\begin{align}
    \tilde K_{a,b}(u;m) = \frac{\Gamma(\tilde d_m)}{\Gamma(a+b)}
    \pFq21(\tilde d_m,b;a+b;u), 
\label{eq:hgeofun-defn-tildeKab}
\end{align}
where $\tilde d_m \defeq \tilde d_m(a,b)=a+b+m/2-2$. Similarly,
\begin{align}
    \tilde H_{a,b,c}(u,v;m) = \frac{\Gamma(\tilde d_m)}{\Gamma(a+b+c)}
    F_1(\tilde d_m;c,b;a+b+c;u,v)
\label{eq:hgeofun-defn-tildeHabc}
\end{align}
with $\tilde d_m \defeq \tilde d_m(a,b,c) = a+b+c+m/2-2$.
\begin{prop}
    For the one variable family defined in \eqref{eq:hgeofun-defn-tildeKab}, 
    the following functional relations hold:
    \begin{align}
        \tilde K_{a,b}(u;m+2) &= (\tilde d_m+ ud/du) 
        \tilde K_{a,b}(u;m) ,
        \label{eq:hggeofun-DR-tildeKanm+2}
        \\
        \tilde K_{a,b+1}(u;m) &= (b^{-1}d/du) \tilde K_{a,b}(u;m) ,
        \label{eq:hggeofun-DR-tildeKanb+1}
        \\
        \tilde K_{a,b+1}(u;m) &= (1+b^{-1}ud/du) \tilde K_{a+1,b}(u;m) 
    \label{eq:hggeofun-DR-tildeKanb+1-a+1}.
    \end{align}
    Moreover,
    \begin{align}
\label{eq:hggeofun-DR-tildeKan-GCF}
        \frac{ a+b}{ \tilde d_m} \frac{
            K_{a+1,b}(u;m)
        }{K_{a,b}(u;m)}  = 
        \frac{
            \pFq21(\tilde d_m+1,b;a+b+1;u)
        }{
        \pFq21(\tilde d_m,b;a+b;u)
    },
    \end{align}
    where the right hand side is a Gauss's continued fraction.
\end{prop}
\begin{proof}
    Follows quickly from the differential relations
    \eqref{eq:hgfun-diff-relations-GF21-abc+1} to
    \eqref{eq:hgfun-diff-relations-GF21-c-1}.
\end{proof}

\begin{prop}
    For the two variable family defined in \eqref{eq:hgeofun-defn-tildeHabc},
    \begin{align}
    H_{a,b,c}(u,v;m+2) &= (\tilde d_m + u\partial_u +v\partial_v)
    H_{a,b,c}(u,v;m),
        \label{eq:hpgeofun-DR-Habcm+2} \\
        H_{a,b+1,c}(u,v;m) &= b^{-1} \partial_v H_{a,b,c}(u,v;m),    
        \label{eq:hpgeofun-DR-Habc-b+1} \\
        H_{a,b,c+1}(u,v;m) &= c^{-1} \partial_u H_{a,b,c}(u,v;m) .   
        \label{eq:hpgeofun-DR-Habc-c+1} 
    \end{align}
    When increasing the parameter $a$ by one, we encounter a similar ration as in
\eqref{eq:hggeofun-DR-tildeKan-GCF}:
    \begin{align*}
        \frac{
        H_{a+1,b,c}(u,v;m)
    }{ H_{a,b,c}(u,v;m) } = 
    \frac{\tilde d_m}{ a+b+c}
    \frac{
        F_1(\tilde d_m+1;c,b;a+b+c+1;u,v)
    }{
        F_1(\tilde d_m;c,b;a+b+c;u,v) 
    } .
    \end{align*}
\end{prop}
\begin{proof}
    Follows quickly from the differential relations
\eqref{eq:hgeofun-f1abc+-diffsys} to
\eqref{eq:hgeofun-f1c+-diffsys}.
\end{proof}

\begin{prop}
        \label{prop:hgfun-relations-K-H} 
    Let $m=\dim M \ge 2$ and $a,b,c$ be positive integers,
    \begin{align}
        \tilde K_{a,b}(z;m+2) &= a \tilde K_{a+1,b}(z;m)
        + b\tilde K_{a,b+1}(z;m)
        \label{eq:hgfun-relations-K-m+2-a+1-b+1} \\
        \begin{split}
         \tilde H_{a,b,c}(u,v;m+2) &= a\tilde H_{a+1,b,c}(u,v;m)
 +b\tilde H_{a,b+1,c}(u,v;m) \\ 
 &+  c\tilde H_{a,b,c+1}(u,v;m).
   \end{split}
                \label{eq:hgfun-relations-H-m+2-a+1-b+1} 
    \end{align}
\end{prop}
\begin{proof}
    Notice that
    \eqref{eq:hgfun-relations-K-m+2-a+1-b+1} and
    \eqref{eq:hgfun-relations-H-m+2-a+1-b+1} are equivalent to the following
    contiguous relations of hypergeometric functions respectively:
  \begin{align}
\begin{split}
       &\,\,  (a+b) \pFq21(\tilde d_m+1,b;a+b;u)\\
        =&\,\,
        a \pFq21(\tilde d_m+1,b;a+b+1;u) +
        b \pFq21(\tilde d_m+1,b+1;a+b+1;u),           
        \end{split}
        \label{eq:hgeofun-evenm-tildeKab-inductive-step-v2}
    \end{align}
and
    \begin{align}
        \label{eq:hgeofun-evenm-tildeHabc-inductive-step-v2}
        \begin{split}
            &\,\, (a+b+c) F_1(\tilde d_m +1;c,b;a+b+c;u,v) 
        \\ =&\,\, 
        a F_1(\tilde d_m +1;c,b;a+b+c+1;u,v) \\
        +& \,\, b F_1(\tilde d_m +1;c,b+1;a+b+c+1;u,v) \\
        +& \,\,
        c F_1(\tilde d_m +1;c+1,b;a+b+c+1;u,v).
        \end{split}
    \end{align}
    We prove \eqref{eq:hgeofun-evenm-tildeHabc-inductive-step-v2} as an example
    and left \eqref{eq:hgeofun-evenm-tildeKab-inductive-step-v2} to interested
    readers.
Indeed,  let $F = F_1(\alpha
    ;\beta,\beta';\gamma;u,v)$, from   \eqref{eq:hgeofun-f1b+-diffsys},
    \eqref{eq:hgeofun-f1b'+-diffsys} and
    \eqref{eq:hgeofun-f1c+-diffsys}, we can solve for
    $(u\partial_u+v\partial_v)F$ in two different ways:
 \begin{align*}
     (u\partial_u+v\partial_v)F &= (\gamma-1) (F(\gamma-) - F) \\
&= 
\beta(F(\beta+)-F) + \beta'(F(\beta'+)-F),
    \end{align*}
    where $F(\beta+)$, $F(\beta'+)$ and $F(\gamma-)$ stand for rising or
    lowering the indicated parameter by one. 
     Two sides of  \eqref{eq:hgeofun-evenm-tildeHabc-inductive-step-v2} appear
     as the two lines on the right hand side above, with $\alpha = \tilde
     d_m+1$, $\beta = c$, $\beta'=b$ and $\gamma = a+b+c+1$. 
\end{proof}


\subsection{Comparison to previous results}
The main goal of this section is to confirm that  the family $H_{a,b,c}(u,v;m)$
defined in Proposition \ref{prop:hgeofun-two-var-family} does agree with those
cases that have been computed in previous works. First of all, asking a computer
algebra system to perform symbolic integration for the double integral
\eqref{eq:hgeofun-twovar-H_abcm} is tremendously  inefficient. Thanks to
Proposition \ref{prop:hpgeofun-F2toGF21-thm}, all the $F_1$ functions that come
from $H_{a,b,1}(u,v;m)$ can be reduced to hypergeometric functions $\pFq21$. In
particular, by taking $p=0$ and $q = 1$ in Proposition
\ref{prop:hpgeofun-F2toGF21-thm}, we see that $F_1(a,1,1,b,x,y)$ is a divided difference of $\pFq21$:
\begin{prop}
    \label{prop:hgeofun-cor-F1-dividedidff}
    For $a\in \mathbb{C}$ and $b \in \mathbb{C} \setminus \Z_{\le 0}$, we have
    the following divided difference relation:
    \begin{align}
        \label{eq:hgeofun-cor-F1-dividedidff}
        \begin{split}
        F_1(a;1,1;b;x,y) &= 
        \frac{x  \pFq21 (a,1;b;x)-y \pFq21 (a,1;b;y)}{x-y}
        \\
        &= [x,y](z \pFq21 (a,1;b;z)).    
        \end{split}
    \end{align}
    and
    \begin{align}
\label{eq:hgeofun-cor-F1-for-c=2}
        F_1(a;1,2;b;x,y) &= 
        b^{-1} (x-y)^{-2} \left[ 
        b x^2  \pFq21(a,1;b;x) +b y^2  \pFq21(a,2;b;y) 
    \right.\\& \left.
        + x \left(-a y^2  \pFq21(a+1,2;b+1;y)-2
        b y  \pFq21(a,1;b;y)\right)
        \right]
        \nonumber
    \end{align}
\end{prop}

In dimension $m=\dim M =2$, we can recover the explicit functions listed at the
very end of \cite{MR3194491}. According to  \eqref{eq:hgeofun-cor-F1-dividedidff},
\begin{align*}
 \tilde   H_{a,1,1}(s,t;2) &= \frac{\Gamma (a+1)
     F_1(a+1;1,1;a+2;1-t,1-s)}{\Gamma (a+2)}
\\
&=
\Scale[1.2]{
\frac{\Gamma (a+1) ((s-1)  \pFq21(a+1,1;a+2;1-s)-(t-1) 
 \pFq21(a+1,1;a+2;1-t))}{\Gamma (a+2) (s-t)}
 }.
\end{align*}
For example,
\begin{align*}
    \tilde H_{2,1,1}(s,t;2) &= 
    \frac{(t-1)^2 \log (s)+(s-1) ((t-1) (s-t)-(s-1) \log (t))}{(s-1)^2 (t-1)^2
    (s-t)}, \\
    \tilde H_{3,1,1}(s,t;2) &= \Scale[1.2]{
    \frac{(s-1) (t-1) (s-t) ((s-3) t-3 s+5)+2 (s-1)^3 \log (t)-2 (t-1)^3 \log
    (s)}{2 (s-1)^3 (t-1)^3 (s-t)}
}.
  \end{align*}
  Those functions was computed again in 
\cite[Sec. 5.2, 5.3]{leschdivideddifference} by applying   divided difference
repeatedly on to the  modified logarithm function $\mathcal L_0$ in
\cite{MR2907006}, which, in particular,  is a  hypergeometric function: 
\begin{align*}
    \pFq21(1,1;2;1-z) = \frac{\ln z}{z-1}.
\end{align*}
One also observes that, in \cite[Sec. 5.2, 5.3]{leschdivideddifference},
characteristic differential operators (as in
 \eqref{eq:hgeofun-f1a+-diffsys} to \eqref{eq:hgeofun-f1c+-diffsys}) were used
 in the computation. Such similarity will be investigate in future papers.

In dimension $m=3$, some explicit functions listed  in   in \cite[Theorem
7.1]{2016arXiv161004740K}. The results are compatible,  we give a few
  examples:
  \begin{align*}
      \tilde K_{2,1}(1-z;3) &=
      \frac{3}{8} \sqrt{\pi } \pFq21 \left(\frac{5}{2},1;3;1-z\right)
      =
      \frac{\sqrt{\pi } \left(\sqrt{z}+2\right)}{2 \left(\sqrt{z}+1\right)^2
      \sqrt{z}} ,\\
\tilde K_{2,1}(1-z;3)  &= \frac{5 \sqrt{\pi }}{16}
\pFq21(7/2,1;4;1-z) =
\frac{\sqrt{\pi } \left(3 z+9 \sqrt{z}+8\right)}{8 \left(\sqrt{z}+1\right)^3
\sqrt{z}},
  \end{align*}
and
\begin{align*}
    \tilde H_{2,1,1}(1-x,1-y;3)& = \frac{5 \sqrt{\pi }}{16}
    F_1(\frac{7}{2};1,1;4;1-y,1-x)\\
    &=\Scale[1.1]{
    \frac{\sqrt{\pi } \left(x \sqrt{y}+\sqrt{x} y+4 \sqrt{x} \sqrt{y}+2 x+4
    \sqrt{x}+2 y+4 \sqrt{y}+2\right)}{2 \left(\sqrt{x}+1\right)^2 \sqrt{x}
    \left(\sqrt{y}+1\right)^2 \sqrt{y} \left(\sqrt{x}+\sqrt{y}\right)}
},
\\
\tilde H_{1,1,1}(1-x,1-y;3)& =   \frac{5 \sqrt{\pi }}{16}
F_1(\frac{5}{2};1,1;3;1-y,1-x) \\
&= 
\frac{\sqrt{\pi } \left(\sqrt{x}+\sqrt{y}+1\right)}{\left(\sqrt{x}+1\right)
\sqrt{x} \left(\sqrt{y}+1\right) \sqrt{y} \left(\sqrt{x}+\sqrt{y}\right)}.
\end{align*}
It has been shown in the author's previous work \cite[Eq. (3.9)] {LIU2017138}
that for dimension $m = \dim M > 2$ and even, functions $\tilde K_{a,b}(u;m)$
and $\tilde H_{a,b,c}(u,v;m)$ are jets of the following functions at zero:  Eq.
\eqref{eq:hgeofun-evenm-tildeKab} and  \eqref{eq:hgeofun-evenm-tildeHabc}. 
Using the recurrence relations in Proposition \eqref{prop:hgfun-relations-K-H},
 we  can give    a short induction proof as below.
\begin{prop}
    Assume that the dimension $m$ is even and greater or equal than $4$, set $j_m
    =(m-4)/2 \in \Z_{\ge0}$, we have 
    \begin{align}
        \tilde K_{a,b}(u;m) &= \frac{d^{j_m}}{dz^{j_m}} \Big|_{z=0}
        (1-z)^{-a} (1-u-z)^{-b} ,
        \label{eq:hgeofun-evenm-tildeKab} \\
        \tilde H_{a,b,c}(u,v;m) &=
        \frac{d^{j_m}}{dz^{j_m}} \Big|_{z=0}
        (1-z)^{-a} (1-u-z)^{-b}(1-v-z)^{-c}.
        \label{eq:hgeofun-evenm-tildeHabc} 
    \end{align}
\end{prop}
\begin{proof}
    We use induction on the dimension $m$. 
    Let us focus on the one variable case first. 
    When $m =4$, $\tilde d_m = a+b$,
    \eqref{eq:hgeofun-evenm-tildeKab} follows from the identity:
    \begin{align*}
        \pFq21(\alpha,\beta;\beta,z) = \pFq21(\beta,\alpha;\beta,z)
        = (1-z)^{-\alpha}.
    \end{align*}
    Now assume that \eqref{eq:hgeofun-evenm-tildeKab} holds for some even $m$,
    it remains to show that 
    \begin{align}
        \tilde K_{a,b}(u;m+2)& = \frac{d^{j_m}}{dz^{j_m}} \Big|_{z=0}
        \frac{d}{dz}\sbrac{ (1-z)^{-a} (1-u-z)^{-b} }\nonumber \\
        &= a \tilde K_{a+1,b}(u;m) + b \tilde K_{a,b+1}(u;m),    
        \label{eq:hgeofun-evenm-tildeKab-inductive-step}
    \end{align}
    which is valid due to Proposition \eqref{prop:hgfun-relations-K-H}.
    Therefore \eqref{eq:hgeofun-evenm-tildeKab} has been proved by induction. Same arguments work for \eqref{eq:hgeofun-evenm-tildeHabc}.

\end{proof}

\section{Variational Calculus with respect to a noncommutative variable}
\label{sec:varcalwrptncvar}
\subsection{Notations}
The calculation  in this section can be carried out in an abstract setting as
in \cite{leschdivideddifference}. Nevertheless, to smooth the exploration, we
take the ambient $C^*$-algebra to be a noncommutative two torus
$C(\T_\theta^2)$ for some irrational $\theta$. Denote by $\nabla$ the
Levi-Civita connection for the flat metric, and then $i\nabla_j$ for $j=1,2$
correspond to the basic derivations, see \cite[Sect.
1.3]{MR3194491}. 


We start with  a fixed self-adjoint element 
$h = h^*$  whose exponential $k = e^h$ defines a Weyl factor. Consider the
variation $\delta_a$ along another self-adjoint operator $a$: 
\begin{align}
    k_\varepsilon \defeq k(\varepsilon) = e^{h+\varepsilon a}, \,\,\,\,
    \delta_a  \defeq \frac{d}{d\varepsilon}\Big|_{\varepsilon=0}. 
    \label{eq:varprob-defn-delta_a} 
\end{align}

The modular operator and the modular
 derivation are denoted by  
 bold letters $\mathbf y$ and $\mathbf x$ respectively:
\begin{align}
    \label{eq:varprob-modop-modder}
  \mathbf  y = \op{Ad}_k = k^{-1}(\cdot) k, \,\,\, \mathbf x = \log
  \mathbf y = -\op{ad}_h = [\cdot,h].
\end{align}
They should be compared with  the regular font $y$ and $x$, which stand for the
positive and real variable, respectively, for
spectral functions\footnote{
    In \cite{LIU2017138}, modular operators and modular derivations are denoted
    by $\modop$ and $\logmodop$ respectively. In this section, they will be
    treated as arguments of spectral functions in lengthy variational calculus, 
    therefore it should be more natural to denote them as letters and use the
    bold font whenever we need to emphasize their operator nature.
}.  Also, as explained in section
\ref{sec:notations-from-Les17}, when acting on
a product of $n$ factors: $\rho_1 \cdots \rho_n$,   we shall  denote by $ 
\mathbf y^{(j)} $, $\mathbf  x^{(j)}$ or simply $\mathbf y_j$, $\mathbf x_j$ to
indicate that the operator acts only on the $j$-the factor, where
$j=1,\dots,n$. To make it more explicit, 
consider a multi-variable  function $K(x_1,\dots,x_n)$  which admits
a Fourier transform:
\begin{align*}
    K(x_1,\dots,x_n) = \int_{\R^n} \alpha(\xi_1,\dots,\xi_n) e^{-i (x_1\xi_1 + 
            \cdots  + x_n\xi_n
    )} d\xi_1\cdots\xi_n,
\end{align*}
then the functional calculus $K(\mathbf x_1,\dots,\mathbf x_n)$ has the
following description 
\begin{align}
    K(\mathbf x_1,\dots,\mathbf x_n) = \int_{\R^n}
\alpha(\xi_1,\dots,\xi_n) e^{-i (\mathbf x_1\xi_1 + 
            \cdots  + \mathbf x_n\xi_n
    )} d\xi_1\cdots d \xi_n.
    \label{eq:varprob-defn-Kx1xn}
\end{align}
To define $K(\mathbf y_1,\dots,\mathbf y_n)$, we need to assume 
$H(x_1,\dots,x_n) \defeq K(e^{x_1},\dots,e^{x_n})$ admits a Fourier transform:
\begin{align*}
    H(x_1,\dots,x_n) = \int_{\R^n} \beta(\xi_1,\dots,\xi_n) e^{-i (x_1\xi_1 + 
            \cdots  + x_n\xi_n
    )} d\xi_1\cdots d\xi_n.
\end{align*}
Notice that $e^{i\mathbf x_j \xi_j} = (\mathbf y_j)^{i\xi_j}$, let 
$\rho_1\cdots \rho_n$ be a product, or an elementary tensor in the $C^*$-algebra,
we have,
\begin{align}
    K(\mathbf y_1,\dots,\mathbf y_n)\cdot (\rho_1 \cdots \rho_n)
    = \int_{\R^n} \beta(\xi_1,\dots,\xi_n)
    \mathbf y^{-i\xi_1}(\rho_1) \cdots \mathbf y^{-i\xi_n}(\rho_n) 
    d\xi_1\cdots d\xi_n.
    \label{eq:varprob-defn-Ky1yn}
\end{align}
 Let $\varphi_0$ be the canonical trace on $C^\infty(\T^2_\theta)$,
later we will need  
 some integration by parts identities.
 \begin{lem}
     Keep the notations. We have 
 \begin{align}
    \label{eq:varprob-integrationbyparts-onevar}
    \varphi_0\brac{
        \rho_1 \cdot K(\mathbf y_1)(\rho_2)
    } &=
    \varphi_0\brac{
        K(\mathbf y_1^{-1})(\rho_1) \cdot \rho_2
    }  ,\\
    \varphi_0\brac{
        K(\mathbf y_1,\mathbf y_2)(\rho_1 \cdot \rho_2)
    } &=
    \varphi_0\brac{
        [K(\mathbf y, \mathbf y^{-1}) (\rho_1) ] \cdot \rho_2
    },
    \label{eq:varprob-integrationbyparts-I-I}
    \\
    \label{eq:varprob-integrationbyparts-I}
    \varphi_0\brac{
        \rho_1\cdot K(\mathbf y_1, \mathbf y_2)(\rho_2\cdot\rho_3)
    } &= \varphi_0 \brac{
    K(\mathbf y_2, \mathbf y_1^{-1} \mathbf y_2^{-1})(\rho_1 \cdot \rho_2)\cdot
\rho_3)
    } ,\\
\varphi_0\brac{
     K(\mathbf y_1, \mathbf y_2)(\rho_1\cdot\rho_2) \cdot \rho_3
 } &=  \varphi_0 \brac{
     \rho_1 \cdot K(\mathbf y_1^{-1} \mathbf y_2^{-2}, \mathbf y_1)
     (\rho_2\cdot \rho_3)
 }.
    \label{eq:varprob-integrationbyparts-II}
\end{align}
    
 \end{lem}
 \begin{proof}
     We will only check \eqref{eq:varprob-integrationbyparts-I-I} and leave the
     rest to the  reader. Let $\alpha(u,v)$ be the Fourier transform of the function
     $H(x_1, x_2) \defeq K(e^{x_1},e^{x_2})$ defined as before, then
     \begin{align*}
         \varphi_0\brac{
         K(\mathbf y_1,\mathbf y_2)(\rho_1 \cdot \rho_2)
     }& = \varphi_0\brac{\int_{\R^2} 
     \alpha(u,v) \mathbf y^{-iu}(\varphi_1) \mathbf y^{-iv}(\rho_2)  dudv
 }
      \\ &=\varphi_0\brac{
\sbrac{ \int_{\R^2} 
\alpha(u,v) \mathbf y^{-i(u-v)}(\rho_1) du dv 
} \cdot \rho_2
}
\\ &= \varphi_0\brac{
    \sbrac{
        K(\mathbf y,\mathbf y^{-1})(\rho_1)
    }\cdot \rho_2
}.
     \end{align*}
     Notice that we have used the trace property of $\varphi_0$ 
     in order to reach  the second step. 
 \end{proof}
 
\subsection{Variational Calculus} 
In the previous work of modular scalar curvature
\cite{MR3540454,MR3194491,MR3369894,LIU2017138,2016arXiv161109815C},
variation was carried out with respect to  $h = \log k$, which can be think of
as a noncommutative coordinate in the tangent space of the moduli space of
metrics. In this section,  we would like to perform  parallel computations with
respect
to the coordinate $k$, so that   the spectral functions (called local curvature
functions in previous works) can be simplified   by getting rid of the spectral
functions  (cf. \cite[Eq. (6.9), (6.10)]{MR3194491}) arising   from differentiating
$e^h$, such as 
   $\nabla e^h$ and $\nabla^2 e^h$.

 It has been pointed out in \cite{leschdivideddifference} that 
 divided difference plays a crucial role in such variational calculus. Recall 
\begin{align}
    [x,y](T(z)) \defeq \frac{T(x) - T(y)}{ x - y}. 
    \label{eq:varprob-defn-divideddiff}
\end{align}
For example, $[x , y] (z^m)  = (y^m - x^m)/(x-y)$.

The main goal of this section is to derive the  variation of the following
local expression
\begin{align}
    \delta_a  \varphi_0 \brac{
    k^j  T(\mathbf y)(\nabla k) \cdot \nabla k
    }, \,\,\,\, \forall j\in\R.
    \label{eq:var-EHfun}
\end{align}
where $\varphi_0$ is the canonical trace on noncommutative two tori.
The final result is given in Theorem \ref{thm:varprob-var-localT}.
\begin{lem}
    \label{lem:varprob-nabla-k^j}
    Let $k = e^h$ with $h$ self-adjoint. For $j  \in\R$, 
 \begin{align}
    \label{eq:varprob-nabla-k^j}
     \nabla k^j = k^{j-1} \left( [1,\mathbf y] (z^j)  \right) (\nabla k)
     = k^{j-1 }\frac{\mathbf y^j-1}{\mathbf y-1} (\nabla k).
 \end{align}
  The same identity holds when the covariant differential $\nabla$ is replaced
 by the variation $\delta_a$ differential.
\end{lem}
\begin{proof}
     Recall the notation of the modular derivation $\mathbf x$ and modular
     operator $\mathbf y$ in
     \eqref{eq:varprob-modop-modder}. 
    Equation \eqref{eq:varprob-nabla-k^j}  is  another   version of
    the formula (cf. \cite[Sect. 6]{MR3194491} and \cite{leschdivideddifference})
     \begin{align*}
         \nabla e^h = e^h \frac{ e^{\mathbf x} -1}{ \mathbf x}\nabla h
     \end{align*}
    by changing the variable:
    \begin{align*}
        \nabla \log k  = k^{-1}\frac{\log \mathbf y}{ \mathbf y-1} \nabla k .
    \end{align*}
    Indeed, for $k^j = e^{jh}$,
     \begin{align*}
         \nabla k^j &= \nabla e^{jh} = e^{jh}\frac{e^{j\mathbf x}-1 }{j\mathbf x}
         \nabla (jh) 
         = e^{jh}
         \frac{e^{j\mathbf x}-1 }{j\mathbf x} (j k^{-1}) 
\frac{\log \mathbf y}{ \mathbf y-1} \nabla k \\
&= k^{j-1 }\frac{\mathbf y^j-1}{\mathbf y-1} (\nabla k).
     \end{align*}
 \end{proof}

 The variation $\delta_a( T(\mathbf y))$ can be computed by the following
 Taylor expansion. We refer the proof to \cite{leschdivideddifference}.
\begin{prop} [\cite{leschdivideddifference}, Prop. 3.11]
    \label{prop:varprob-Texp-in-logk}
Let $h,b$ be two self-adjoint elements and $T(x)$ be a Schwartz function on
$\R$. We have the Taylor expansion for the operator $T(\op{ad}_{h+b})$ upto
the first order:
\begin{align} 
    \label{eq:varprob-Texp-in-logk}
    T(\op{ad}_{h+b})(\psi)
    & = T(\mathbf x)(\psi) 
- ([\mathbf x_1+ \mathbf x_2, \mathbf x_2]T)  (b \cdot \psi) 
\\
&+ ([\mathbf
    x_1 + \mathbf x_2, \mathbf x_1]T)  (\psi \cdot b) +  o(b),
    \nonumber
    \end{align}
    as $b \rightarrow 0$ and for all $\psi$ in the ambient $C^*$-algebra.
\end{prop}
It has an exponential version:
\begin{prop}
     \label{prop:varprob-Texp-in-k}
    Let $T(y)$ be a Schwartz function on
    $\R^+$, denote
\begin{align}
        D(T)(y_1, y_2) = y_1 \left[ y_1 y_2, y_1 \right] T
        = \frac{T(y_1 y_2) - T(y_1)}{y_2 -1}. 
        \label{eq:varprob-defn-DT}
    \end{align}
    Then
    \begin{align}
\label{eq:varprob-Texp-in-k}
        \delta_a  (T(\mathbf y))(\nabla k)    &= 
        k^{-1}
        \left[ \mathbf y_1^{-1} D(T)(\mathbf y_1,\mathbf y_2) \right]
        \cdot ( (\nabla k) 
        \delta_a  (k)) \\
        & \,\,\,\,\, \, - 
        k^{-1} 
        D(T)(\mathbf y_2, \mathbf y_1) 
        \cdot (\delta_a  (k) \nabla k)
        \nonumber.
    \end{align}
    Same identity holds when $\delta_a$ is replaced by the covariant
    differential $\nabla$.    
\end{prop}
\begin{proof}
     Denote $\tilde T(x) = T(e^x)$, then
    \begin{align*}
        \delta_a(T(\mathbf y)) = \frac{d}{d\varepsilon} \Big|_{\varepsilon=0}
        \tilde T(\op{ad}_{h+\varepsilon a}).
    \end{align*}
    Eq.  \eqref{eq:varprob-Texp-in-k} follows from
    \eqref{eq:varprob-Texp-in-logk} after some substitutions. 
\end{proof}
Now we are ready to compute the commutator of the covariant differential
$\nabla$ and the modular action $T(\mathbf y)$.
\begin{lem}
    \label{lem:varprob-swapping-nabla-x}
    Keep the notations. 
\begin{align*}
        \nabla \left[ T(\mathbf y)(\nabla k) \right] - T(\mathbf y) [ \nabla^2
        k] = 
        k^{-1} \left[ \mathbf y_1^{-1} D(T)(\mathbf y_1, \mathbf y_2)
        - D(T)(\mathbf y_2, \mathbf y_1) \right] (\nabla
        k  \nabla k),
    \end{align*}
    where the function $D(T)$ is defined in {\rm Prop.}
    $\ref{prop:varprob-Texp-in-k}$. 
    \end{lem}
    \begin{proof}
        Use the Leibniz property,
        \begin{align*}
         \nabla \left[ T(\mathbf y)(\nabla k) \right]  = 
         \nabla [T(\mathbf y)] (\nabla k) + T(\mathbf y)(\nabla^2 k),
        \end{align*}
       where the first term  $\nabla [T(\mathbf y)] (\nabla k)$ has been
       computed in  \eqref{eq:varprob-Texp-in-k}.
    \end{proof}

    Now we are fully prepared  for computing the variation $\delta_a[\varphi_0\brac{
        k^j T(\mathbf y)(\nabla k)\cdot (\nabla k)
}]$, which consists of four parts by the Leibniz  property: 
\begin{align*}
& \,\,\varphi_0\brac{
    \delta_a[k^j]T(\mathbf y)(\nabla k)\cdot (\nabla k) 
} + \varphi_0\brac{
    k^j \delta_a[T(\mathbf y)](\nabla k)\cdot (\nabla k) 
} \\
 +& \,\,\varphi_0\brac{
    k^j T(\mathbf y)( \delta_a[\nabla k]) \cdot (\nabla k)
} +
\varphi_0\brac{ 
    k^j T(\mathbf y)(\nabla k)\cdot (\delta_a[\nabla k])
}.
\end{align*}
We split the calculation into four lemmas and summarize the result at the end.
\begin{lem}
        \label{lem:varprob-lemma1}
    Keep the notations. 
    \begin{align}
        \label{eq:varprob-lemma1}
        &\,\,\varphi_0 \left( \delta_a  (k^j) T(\mathbf y)(\nabla k) \cdot
            \nabla
        k \right)\\
        =&\,\, \varphi_0\brac{
            \delta_a  (k) k^{j-1} ([1,\mathbf y_1 \mathbf y_2]\op{id}^j)^{-1}
            (\mathbf y_1
            \mathbf y_2)^{j-1} T(\mathbf y_1)(\nabla k \cdot \nabla k)
        } 
        \nonumber \\ 
        \defeq &\,\,  \varphi_0\brac{
            \delta_a(k) k^{j-1}  \mathrm{II}^{(1)} (\mathbf y_1, \mathbf y_2)
            (\nabla k \cdot \nabla k)
        }, \nonumber
    \end{align}
    where
    \begin{align*}
        \mathrm{II}^{(1)}(T)(y_1, y_2) = [1,y_1^{-1} y_2^{-1}]\op{id}^j  (y_1
        y_2)^{j-1} T(y_1).
    \end{align*}
\end{lem}
\begin{proof}
    Apply \eqref{eq:varprob-nabla-k^j} onto $\delta_a(k^j)$, we have 
    \begin{align*}
        &\,\,      \varphi_0\brac{
    \delta_a[k^j]T(\mathbf y)(\nabla k)\cdot (\nabla k) 
} \\
=&\,\,
\varphi_0\brac{
\mathbf y^{1-j} \frac{\mathbf y^j -1}{\mathbf y-1} [\delta_a(k)] k^{j-1}
T(\mathbf y)(\nabla k) \cdot (\nabla k)
}.
    \end{align*}
    Then Eq. \eqref{eq:varprob-lemma1}  follows  from integration by parts:
    \eqref{eq:varprob-integrationbyparts-II}.
\end{proof}

\begin{lem}
    Keep notations. We have
    \begin{align*}
        \varphi_0 \brac{k^j \delta_a  \left[ T(\mathbf y) \right] (\nabla k)
        \cdot \nabla k }
         = 
         \varphi_0 \brac{ \delta_a(k) k^{j-1}
         \mathrm{II}^{(2)}(T;j)(\mathbf y_1, \mathbf y_2)  (\nabla k \cdot \nabla
     k) },
    \end{align*}
    with
    \begin{align*}
        \mathrm{II}^{(2)}(T;j)(y_1, y_2) = y_2^{-1} y_1^{j-1} D(T)(y_2, y_1^{-1}
        y_2^{-1} ) -  (y_1 y_2)^{j-1} D(T)(y_1, y_1^{-1} y_2^{-1}).
    \end{align*}
    where $D(T)$ is defined in {\rm Prop.} $\ref{prop:varprob-Texp-in-k}$.
\end{lem}
\begin{proof}
    We substitute $\delta_a  \left[ T(\mathbf y) \right] (\nabla k)$ for the
    right hand side of \eqref{eq:varprob-Texp-in-k}:
\begin{align*}
        \varphi_0 \brac{  k^j \delta_a \left[  T(\mathbf y) \right]
        (\nabla k) \cdot \nabla k} &=
        \varphi_0 \brac{  k^{j-1} \left[ \mathbf y_1^{-1} D(T)(\mathbf
            y_1,\mathbf y_2) \right]
            (\nabla k \cdot \delta_a(k)) \cdot \nabla k
        } \\
        &- \varphi_0 \brac{k^{j-1}
        D(T)(\mathbf y_2,\mathbf y_1) (\delta_a(k) \nabla k) \cdot
        \nabla k
    }.
    \end{align*}
    It remains to bring $\delta_a(k)$ in front and move the modular operators
    onto $\nabla k$ by the integrations by parts identities
    \eqref{eq:varprob-integrationbyparts-I} and
    \eqref{eq:varprob-integrationbyparts-II}. For the first term:
     \begin{align*}
        &\,\, \,  \varphi_0 \brac{  k^{j-1} \left[ \mathbf y_1^{-1}
                D(T)(\mathbf y_1,\mathbf y_2)
            \right] (\nabla k \cdot \delta_a(k)) \cdot \nabla k
    } \\ =  &\,\, \, \varphi_0 \brac{ \left( \nabla k \cdot  k^{j-1} \right)
        \left[ \mathbf y_1^{-1} D(T)(\mathbf y_1,\mathbf y_2) \right] (\nabla
            k \cdot
        \delta_a(k)) 
        } \\
        =  &\,\, \, \varphi_0 \brac{ \left( \mathbf y_2^{-1} D(T)(\mathbf y_2,
                    \mathbf y_1^{-1}
        \mathbf y_2^{-1}) \right) \left[ \left( \nabla k \cdot  k^{j-1}
    \right)\cdot\nabla k \right]  \delta_a(k)) 
        } \\
        =  &\,\, \, \varphi_0 \brac{
            \delta_a(k) k^{j-1}  \left( \mathbf y_2^{-1} \mathbf
                y_1^{j-1} D(T)(\mathbf y_2,
            \mathbf y_1^{-1} \mathbf y_2^{-1} ) \right)\left[  \nabla k \cdot
            \nabla k \right]
        }.
    \end{align*}
For the second term:
\begin{align*}
     &\,\, \,  \varphi_0 \brac{  k^{j-1}
      D(T)(\mathbf y_2, \mathbf y_1)  (\delta_a(k) \nabla k) \cdot \nabla k
     } \\ = &\,\,\,
     \varphi_0 \brac{ \delta_a(k)
         D(T)(\mathbf y_1, \mathbf y_1^{-1} \mathbf y_2^{-1}) \left[ \nabla
         k \cdot (\nabla k \cdot k^{j-1}) \right]
     } \\ = &\,\,\,
     \varphi_0 \brac{ \delta_a(k) k^{j-1}
         \left( (\mathbf y_1 \mathbf y_2)^{j-1} D(T)(\mathbf y_1, \mathbf
         y_1^{-1} \mathbf y_2^{-1}) \right) (\nabla k \cdot \nabla k)
     }.
 \end{align*}

 \end{proof}

\begin{lem}
    \label{lem-varprob-lemma3}
    Keep notations,   
       \begin{align*}
        \varphi_0 \brac{
        k^j T(\mathbf y)\left( \nabla k \right) \cdot \delta_a  (\nabla k)  
    } = & \,\, \varphi_0 \brac{
         \delta_a  (k) k^{j} \op{I}^{(2)}(T)(\mathbf y) (\nabla^2 k) 
    }\\
    &\,\,  + \varphi_0 \brac{
       \delta_a  (k)   k^{j-1} \op{II}^{(4)}(T)(\mathbf y_1, \mathbf y_2) (\nabla k \cdot \nabla k) 
    },
    \end{align*}
    while 
    \begin{align*}
        \op{I}^{(1)}(T)(y) &= - T(y),
        \\
\op{II}^{(3)}(T)(y_1, y_2) &= -y_1^{-1} D( T) (y_1, y_2)
          +   D( T) (y_2, y_1) \\
          & - ([1,y_1]\op{id}^j)  T(y_2).
    \end{align*}
\end{lem}

The proof of Lemma \ref{lem-varprob-lemma3} and Lemma \ref{lem-varprob-lemma4} are quite
similar. Thus we only prove the more complicated one (Lemma
\ref{lem-varprob-lemma4}) as an exmaple.

\begin{lem}
    \label{lem-varprob-lemma4}
    Keep notations. Set $\mathcal I(T;j)(y) = y^j T(y^{-1})$ with $j \in \R$.
    \begin{align*}
        \varphi_0 \brac{
        k^j T(\mathbf y)(\delta_a  (\nabla k)) \cdot \nabla k 
    } &= \varphi_0 \brac{
         \delta_a  (k) k^{j} \op{I}^{(1)}(T)(\mathbf y) (\nabla^2 k) 
    }\\
    &+ \varphi_0 \brac{
        \delta_a  (k)  k^{j-1} \op{II}^{(3)}(T)(\mathbf y_1, \mathbf y_2) (\nabla k \cdot \nabla k) 
    },
    \end{align*}
    while
    \begin{align*}
        \op{I}^{(2)}(T)(  y) &= - \mathcal I(T;j)(y), \\
        \op{II}^{(4)} (T)(  y_1,   y_2)   &= 
        \op{II}^{(3)} (\mathcal I(T;j))(  y_1,   y_2) 
        \\ &=
        -y_1^{-1} D( \mathcal I(T;j) )
        (  y_1,   y_2)
          +   D( \mathcal I(T;j)) (  y_2,   y_1) \\
          & - ([1,  y_1]\op{id}^j)  \mathcal I(T;j)(  y_2) .
    \end{align*}
\end{lem}
\begin{proof}
    Use the trace property and the fact that $\delta_a$ commutes with $\nabla$:
    \begin{align*}
        &\,\, \, \varphi_0 \brac{
        k^j T(\mathbf y)(\delta_a(\nabla k)) \cdot \nabla k 
        }
        \\ = & \,\,\,
     -   \varphi_0 \brac{
         \delta_a(k) \nabla  \left[ T(\mathbf y^{-1})(\nabla k) k^j \right]
    }
        \\ = & \,\,\,
      -   \varphi_0 \brac{
         \delta_a(k) \nabla  \left[k^j (\mathbf y^j T(\mathbf y^{-1}))(\nabla
         k) \right]
    }
    \end{align*}
    in which, 
\begin{align*}
    \nabla  \left[k^j (\mathbf y^j T(\mathbf y^{-1}))(\nabla k) \right]
 = (\nabla k^j) \cdot (\mathbf y^j T(\mathbf y^{-1}))(\nabla k) + 
 k^j \nabla  \left[ (\mathbf y^j T(\mathbf y^{-1}))(\nabla k) \right]. 
\end{align*}
For the first term,
\begin{align*}
    (\nabla k^j) \cdot (\mathbf y^j T(\mathbf y^{-1}))(\nabla k)  &= 
    k^{j-1} ([1,\mathbf y]  z^j) (\nabla k) \cdot (\mathbf y^j T(\mathbf
    y^{-1}))(\nabla k) \\ &=
 k^{j-1} ([1,\mathbf y_1]  z^j) (\mathbf y_2^j T(\mathbf y_2^{-1}))
 (\nabla k \cdot \nabla k)
\end{align*}
For the second one, we denote $\tilde T(y) = y^j T(y^{-1})$,
\begin{align*}
   k^j \nabla \left[ \tilde{T}(\mathbf y) (\nabla k) \right]  &=  
  k^j  \tilde T(\mathbf y)(\nabla^2 k) + k^{j-1} \left[ \mathbf y_1^{-1} D(\tilde
    T) \right] (\mathbf y_1, \mathbf y_2) (\nabla k \cdot \nabla k) \\
    &-
    k^{j-1} D(\tilde T) (\mathbf y_2, \mathbf y_1) (\nabla k \cdot \nabla k).
\end{align*}

\end{proof}
We group the four lemmas above into our main theorem of this section.
\begin{thm}
    \label{thm:varprob-var-localT}
    For $j \in \R$, let $\mathrm{I}^{(\alpha)}(T)$, $\alpha = 1,2$ and
    $\mathrm{II}^{(\beta)}(T)$, $\beta = 1,2,3,4$, be the operations on the
    function $T$ defined in the previous lemmas, see also \textrm{Eq.}
    \eqref{eq:varprob-IIT-explicit}, then we have the following
    variation formula for the local expression:
    \begin{align*}
    \delta_a  \varphi_0 \brac{
    k^j  T(\mathbf y)(\nabla k) \cdot \nabla k
} &= \varphi_0 \brac{
    \delta_a(k) k^j 
    (\mathrm{I}^{(1)} + \mathrm{I}^{(2)})(T;j) (\mathbf y)(\nabla^2 k)
} \\
&+ \varphi_0 \brac{
\delta_a(k) k^{j-1} ( \sum_{\beta=1}^4 \mathrm{II}^{(\beta)}) (T;j) (\mathbf y_1,
\mathbf y_2) (\nabla k \nabla k)
} .
    \end{align*}
\end{thm}

\subsection{Examples}
The operations $\sum_{\alpha=1}^2\mathrm{I}^{\alpha}(T;j)$ and
$\sum_{\beta=1}^4 \mathrm{II}^{\beta}(T;j)$  are quite
complicated. In fact, $\sum_{\beta=1}^4 \mathrm{II}^{\beta}(T;j)$ is expanded
to thousands of terms when performing simplification in \textsf{Mathematica}.   
More conceptual understanding will be investigated in future papers. 
As we shall see in later computation (such as Eq. 
\eqref{eq:check-TDelta-using-GF21}), typical terms of $T(z)$ are  of the form 
$\pFq21(a,1;c;1-z)$. Let us compute part of $\mathrm{I}(T;j)$ and
$\mathrm{II}(T;j)$ in this case.
There two basic operations on $T(z)$ in the variation. 
One is $\mathcal I(T;j)$, which arises from integration by parts. The other,
$D(T;j)$ in \eqref{eq:varprob-defn-DT},
measures the commutator of the classical covariant differential and the modular
action. 


Let us start with the one variable case, which is related to the Pfaff
transformations \eqref{eq:hgeofun-Pfaff-1}.
\begin{lem}
    \label{lem:varprob-Ttilde}
    Let 
    \begin{align*}
        T(u) = K_{a,b}(u;m) = \frac{
            \Gamma(\tilde d_m)
        }{\Gamma(a+b)} \pFq21(\tilde d_m,b;a+b;1-u),
    \end{align*}
    where $\tilde d_m = a+b-2+m/2$. Then 
    \begin{align}
        \mathcal I(T;j)(u)  \defeq u^j T(u^{-1})= u^{j+b} K_{b,a}(u;m).   
        \label{eq:varprob-ITJ-T=Kab}
    \end{align}
\end{lem}
\begin{proof}
    Let $z = 1-u^{-1}$ in \eqref{eq:hgeofun-Pfaff-1},
    we obtain 
    \begin{align*}
        \pFq21(\tilde d_m,b;a+b;1-u^{-1}) = u^{b} \pFq21(\tilde d_m,a;a+b;1-u),
    \end{align*}
    and \eqref{eq:varprob-ITJ-T=Kab} follows immediately. 
\end{proof}


For the two variable situation, we narrow our attention to functions of the form
$\pFq21(a,1;b;1-z)$.
\begin{lem}
    For $T(z) = \pFq21(a,1;b;1-z) $ where $a \in \mathbb{C} $ and $b \in
    \mathbb{C}\setminus\Z_{\le 0} $,
    \begin{align}
        [s,t] (zT(z))& = 
        [1-t,1-s]( z \pFq21(a,1;c;z) ) + [s,t] ( \pFq21(a,1;c;1-z)
        ) \nonumber\\
        &= F_1(a,1;1,c;1- t, 1-s  ) 
+[s,t]( \pFq21(a,1;c;1-z) )
\label{eq:lem-exmaples-1}
    \end{align}
\end{lem}
\begin{proof}
We abbreviate $ \pFq21(a,1;b;z )$ to $F(z)$ in  the calculation:
    \begin{align*}
        [s,t] (zT(z))& = \frac{
            s F(1-s) - t F(1-t)
        }{ s -t} \\& = 
        \frac{
            (s-1) F(1-s) - (t-1) F(1-t) + F(1-s) - F(1-t)
        }{s-t} \\ &=
        \frac{
            (1-t)F(1-t) - (1-s)F(1-s)
        }{
            (1-t) - (1-s)
        } +
        \frac{
            F(1-s) - F(1-t)
        }{ 
            s-t
        } \\ &=
        [1-t,1-s]( z \pFq21(a,1;c;z) ) + [s,t] ( \pFq21(a,1;c;1-z) ) .
    \end{align*}
    The first term above gives rise to the Appell $F_1$ function appeared in
    \eqref{eq:lem-exmaples-1} by Proposition
    \ref{prop:hgeofun-cor-F1-dividedidff}.

\end{proof}
%
  In contrast to Lemma \ref{lem:varprob-Ttilde}, we use a different Pfaff
  transformation in \eqref{eq:hgeofun-Pfaff-1} for $\pFq21(a,1;c;1-z^{-1})$:
  \begin{align}
    \pFq21(a,1;c;1-z^{-1})  = z \pFq21(c-a,1;c;1-z).
    \label{eq:varprob-zto1overz-GF21}
\end{align}
Notice that the function $\pFq21(c-a,1;c;1-z^{-1})$ does not belong to the family 
  $H_{a,b,c}(u,v;m)$.

\begin{lem}
     Let $T(z) = \pFq21(a,1;c;1-z) $ where $a \in \mathbb{C} $ and $b \in
     \mathbb{C}\setminus\Z_{\le 0} $.  For fixed $j\in\R$, denote $\tilde T(z)
     \defeq \mathcal I(T;j) =  z^j T(1/z)$, we have
     \begin{align}
          [s,t] (z \tilde T(z)) 
         &= t \brac{
         [s,t](z^{j+1}) }   \pFq21(c-a,1;c;1-t)
         \\
         &+ s^{j+1} F_1(c-a,1;1,c;1- t, 1-s  )  + s^{j+1} [s,t](
         \pFq21(c-a,1;c;1-z) )
          \label{eq:varprob-lemma-DTTilde}
     \end{align}
\end{lem}
\begin{proof}
    According to   \eqref{eq:varprob-zto1overz-GF21},  $\tilde T$ is a product
    \begin{align*}
        \tilde T(z) = z^j \cdot z  \pFq21(c-a,1;c;1-z).
    \end{align*}
    Use the Leibniz  property for divided differences:
    \begin{align*}
         [s,t] (z \tilde T(z)) 
          & = \brac{
         [s,t](z^{j+1}) } \cdot t \pFq21(c-a,1;c;1-t) \\
          &+ s^{j+1}\sbrac{
              [s,t](z \pFq21(c-a,1;c;1-z)) 
         }
    \end{align*}

\end{proof}

\begin{lem}
    For $T(z) = \pFq21(a,1;c;1-z) $ where $a \in \mathbb{C} $ and $b \in
    \mathbb{C}\setminus\Z_{\le 0} $, 
    \begin{align}
        \label{eq:varprob-DT-GF21}
        \begin{split}
         D(T) (y_1, y_2) 
         &= F_1(a,1;1,c;1- y_1y_2, 1-y_1  ) 
        \\
        & +[y_1,y_1y_2]( \pFq21(a,1;c;1-z) ) - \pFq21(a,1;c;1-y_1y_2) .
   \end{split}
            \end{align}
            Similarly, for $\tilde T(z) = z^j T(1/z) = z^j
            \pFq21(a,1;c;1-z^{-1})$, 
            \begin{align}
                \begin{split}
                    D(\tilde T) (y_1, y_2) &=
       y_1y_2 \brac{
         [y_1,y_1y_2](z^{j+1}) }  \pFq21(c-a,1;c;1-y_1y_2) \\
        & + y_1^{j+1} F_1(c-a,1;1,c;1- y_1y_2, 1-y_1  ) \\
        & + y_1^{j+1} [y_1,y_1y_2]( \pFq21(c-a,1;c;1-z) ) \\
        &- (y_1y_2)^{j+1}\pFq21(c-a,1;c;1-y_1y_2)
                \end{split}
                \label{eq:varprob-DTTilde}
            \end{align}
\end{lem}
\begin{proof}
    Recall that
    \begin{align*}
        D(T)(y_1,y_2) = y_1 [y_1,y_1y_2](T) = [y_1,y_1y_2](zT(z)) - T(y_1 y_2) ,
    \end{align*}
    where the first term $[y_1,y_1y_2](zT(z))$ has been computed in \eqref{eq:lem-exmaples-1}.
\end{proof}

We have another version for Eq. \eqref{eq:varprob-DT-GF21} as below. 
\begin{lem}
    For 
    \begin{align*}
        T(z) = K_{a,1}(z;m) = \frac{ \Gamma(a-1+m/2)}{ a+1}
        \pFq21(a-1+m/2,1,;a+1;1-z),
    \end{align*}
    we have
    \begin{align*}
        D(T)(y_1,y_2) = H_{a-1,1,1}(y_1,y_2)+ [y_1,y_1y_2](
        K_{a,1}(z;m)) - K_{a,1}(y;m) .
    \end{align*}
\end{lem}

\section{Variation of the Einstein-Hilbert action}
\label{sec:var-FEH}

It has been shown in \cite{Liu:2015aa,LIU2017138} 
that the spectral functions appeared
in the modular curvature do not depend on the underlying manifold. Therefore
we shall focus only on smooth noncommutative $m$-tori: $C^\infty(\T^m_\theta)$,
with the flat Euclidean metric as the background metric.
 We refer the construction of the algebra $C^\infty(\T^m_\theta)$ to
 \cite{MR1047281},
and \cite{MR1846904} as a special case of toric noncommutative manifolds.

For the conformal perturbation, we consider the following Laplacian
$\Delta_\varphi = k^{1/2} \Delta k^{1/2}$, which is the even part in the twisted 
spectral triple studied in \cite{MR3194491}. The associated modular curvature
$R_{\Delta_\varphi}$ is   the functional density of the second heat
coefficient(cf. \eqref{eq:intro-heat-expansion}):
\begin{align*}
    V_2(a, \Delta_\varphi) = \varphi_0(a R_{\Delta_\varphi}),
    \,\,\, \forall a \in C^\infty(\T^m_\theta),
\end{align*}
here $\varphi$ denotes the state: $a \in C^\infty(\T^m_\theta) \mapsto 
\varphi_0(a k)$, which is a twisted  trace via the modular operator $\mathbf y$:
\begin{align}
    \label{eq:varprob-volume-varphi}
    \varphi(a b) &= \varphi_0(abk) = \varphi_0(bka) 
    \varphi_0(k \mathbf y(b) a) \\
    &= \varphi_0(\mathbf y(b)a k) = 
    \varphi(\mathbf y(b)a) .
\end{align}

We remind the reader again that  $k = e^h$ is a Weyl  factor
with  $h = h^* \in C^\infty(\T^m_\theta)$
self-adjoint. Bold letters $\mathbf x$ and $\mathbf y$  stand for the modular
derivation and modular operator respectively. Functional calculus, such as
$K(\mathbf y) (\rho)$ and
$H(\mathbf y_1, \mathbf y_2)(\rho_1 \cdot \rho_2)$ (in Eq.
\eqref{eq:varprob-scacur-for-Delta-varphi}), was defined in 
\eqref{eq:varprob-defn-Ky1yn}.

%

Local expression of $R_{\Delta_k}$ was first computed in
\cite{MR3194491,MR3148618}. Later, computation was simplified and generalized to 
all toric noncommutative manifolds by the author in \cite{Liu:2015aa}.
For self-containedness, we sketch the computation in section
\ref{sec:symbloic-computation} following \cite{Liu:2015aa}.
\begin{thm}
    \label{thm:varprob-scacur-for-Delta-varphi}
  Upto a  constant factor $\op{Vol}(S^{m-2})/2$,
$R_{\Delta_\varphi}$ is of the form:
\begin{align}
    \label{eq:varprob-scacur-for-Delta-varphi}
    R_{\Delta_\varphi}  = k^{j_m} K_{\Delta_\varphi} ( \mathbf y)(\nabla^2 k)
    \cdot g^{-1} + k^{j_m-1}  H_{\Delta_\varphi}( \mathbf y_1, \mathbf y_2)
    (\nabla k \nabla k) \cdot g^{-1},
\end{align}
where $j_m = -m/2$ and
\begin{align}
\label{eq:varprob-K-Delta-varphi-vs-k}
K_{\Delta_\varphi} (y;m) &= \sqrt{y} K_{\Delta_k} (y;m), \\
    H_{\Delta_\varphi} (y_1,y_2;m) & = \sqrt{y_1 y_2}  H_{\Delta_k}
    (y_1, y_2;m).
    \label{eq:varprob-H-Delta-varphi-vs-k}
\end{align}
   \end{thm}
   \begin{proof}
       It is easier to deal with $R_{\Delta_k}$ first, where  $\Delta_k
       = k \Delta$, whose calculation is put in section
       \ref{sec:symbloic-computation}. Based on Theorem
       \ref{thm:psecal-R_Delta_k}, the result follows
       quickly from the relation between $\Delta_k$ and $\Delta_\varphi$:      
        $\Delta_\varphi = k^{1/2} \Delta
       k^{1/2} = k^{-1/2} \Delta_k k^{1/2}$. Same argument as in \cite[Eq.
       (3.19)]{LIU2017138} shows that
       \begin{align*}
           R_{\Delta_\varphi} = k^{-1/2} R_{\Delta_k} k^{1/2} 
           = \mathbf y^{1/2}( R_{\Delta_k}) .
       \end{align*}
       The functional relations \eqref{eq:varprob-K-Delta-varphi-vs-k} and
       \eqref{eq:varprob-H-Delta-varphi-vs-k} follows immediately. 
   \end{proof}
\subsection{Functional relations in Higher dimensions}
 After obtaining the 
 explicit expression of $R_{\Delta_\varphi}$, it is natural to consider
 related variation problems  to the
 corresponding Einstein-Hilbert action: 
\begin{align*}
    F_{\mathrm{EH}}(k) = V_2(1,\Delta_\varphi) = \varphi_0( R_{\Delta_\varphi}),
   \end{align*}
   as a functional in $k$, or equivalently, a functional on the conformal class
   of metrics defined by the flat metric.
    Notice that, due to
   \eqref{eq:varprob-V_j-step1},  $F_{\mathrm{EH}}$ is a 
   constant functional  in dimension two. Therefore we always assume that the
   dimension of the underlying manifold is greater or equal than $2$.
\begin{prop}
    \label{prop:varprob-Feh-closedform}
    The Einstein-Hilbert action is given by the following  local formula:
    \begin{align*}
     F_{\mathrm{EH}} (k) =  \varphi_0 \brac{
    k^{j_m-1} T_{\Delta_\varphi} (\mathbf y)(\nabla k) (\nabla k)
        \cdot g^{-1}
    }.
   \end{align*}
where the function $T_{\Delta_\varphi}$ is determined by $K_{\Delta_\varphi}$ and
$H_{\Delta_\varphi}$, 
\begin{align}
    \label{eq:varprob-defn-TDleta}
    T_{\Delta_\varphi}( y ) = -K_{\Delta_\varphi}(1) \frac{y^{j_m} -1}{y-1}
    + H_{\Delta_\varphi}( y,y^{-1}).
\end{align}
\end{prop}
\begin{rem}
    Observe that: 
    \begin{align}
        K_{\Delta_\varphi}(1) = K_{\Delta_k}(1), \,\,\,
        H_{\Delta_\varphi}( y,y^{-1}) =  H_{\Delta_k}( y,y^{-1}) .
        \label{eq:varprob-rem-Kphi=Kk}
    \end{align}
\end{rem}
\begin{proof}
    Since $\varphi_0$ is a trace, we have in general, $\varphi_0(k^{j_m}
    K(\mathbf y)(\rho)) = K(1) \varphi_0(k^{j_m} \rho)$. In particular, 
    \begin{align*}
        \varphi_0\brac{
            k^{j_m} K_{\Delta_\varphi}(\mathbf y)(\nabla^2 k)
        } &= K_{\Delta_\varphi}(1) \varphi_0\brac{
        k^{j_m} \nabla^2 k} \\ &= 
        K_{\Delta_\varphi}(1) \varphi_0 \brac{
            k^{j_m-1} \frac{\mathbf y^{j_m} -1}{ \mathbf y-1} (\nabla k) \cdot 
            (\nabla k)
        }
    ,
    \end{align*}
    here we have used Lemma \ref{lem:varprob-nabla-k^j} for the second ``='' sign.
    This explains the first term in \eqref{eq:varprob-defn-TDleta}. While the
     second term follows immediately  from 
    \eqref{eq:varprob-integrationbyparts-I-I}:
    \begin{align*}
        \varphi_0\brac{
            k^{j_m-1} H_{\Delta_\varphi}(\mathbf y_1,\mathbf y_2)
        (\nabla k \nabla k) } = \varphi_0\brac{
            [H_{\Delta_\varphi}(\mathbf y,\mathbf y^{-1})(\nabla k)]
            \cdot \nabla k
            }
        .
    \end{align*}
\end{proof}

The setting of the variation is identical to the one in the previous section.
For $\varepsilon
> 0$ and another self-adjoint  element $a=a^* \in C^\infty(\T^m_\theta)$, we
consider the family $k_\varepsilon \defeq k(\varepsilon) = e^{h+\varepsilon
a}$ and study the variation 
$\delta_a \defeq d/d\varepsilon |_{\varepsilon = 0}$.

\begin{defn}
    \label{defn:varprob-gradFEH}
    Consider the Einstein-Hilbert functional near a fixed Weyl factor $k$.
     Notice that $a \mapsto \delta_a F_{\mathrm{EH}}(k)$ is
    a linear functional in $a$, we define the gradient of $\grad
    F_{\mathrm{EH}} (k)$ (at $k$) to be the unique element 
    in $C^\infty(\T^m_\theta)$ 
    with the property that
    \begin{align*}
        \delta_a F_{\mathrm{EH}}(k) = \varphi_0 (a \grad F_{\mathrm{EH}}).
    \end{align*}
\end{defn}
Let us start the computation of $\delta_a F_{\mathrm{EH}}(k)$ with
\begin{align}
    \delta_a (\Delta_\varphi) &= \delta_a ( k^{1/2} \Delta k^{1/2})
    = \delta_a(k^{1/2}) \Delta k^{1/2} +  k^{1/2} \Delta \delta_a(k^{1/2})
    \nonumber
    \\ &= 
    (\delta_a(k^{1/2})k^{-1/2}) \Delta_\varphi + \Delta_\varphi
    (k^{-1/2}\delta_a(k^{1/2})) . 
     \label{eq:varpron-defn-fk}
\end{align}
According to Duhamel's formula,  for $t>0$, 
\begin{align*}
    \delta_a \Tr \brac{ e^{-t\Delta_\varphi}} = -t 
    \Tr \brac{
        \delta_a(\Delta_\varphi)  e^{-t\Delta_\varphi}
    }.
\end{align*}
With \eqref{eq:varpron-defn-fk} and the trace property, we continue:
\begin{align*}
    \delta_a \Tr \brac{ e^{-t\Delta_\varphi}} & = -t 
    \Tr \brac{
    (\delta_a(k^{1/2})k^{-1/2} + k^{-1/2}\delta_a(k^{1/2}) )
    \Delta_\varphi e^{-t\Delta_\varphi}
    } \\
    &= t \frac{d}{dt} \Tr \brac{
    (\delta_a(k^{1/2})k^{-1/2} + k^{-1/2}\delta_a(k^{1/2}) )
e^{-t\Delta_\varphi}
    } \\
    & \defeq t \frac{d}{dt} \Tr \brac{
    f_k e^{-t\Delta_\varphi}
}.
\end{align*}
Use the fact that both $\delta_a$ and $d/dt$ can pass through the asymptotic
expansion, we see that 
\begin{align*}
    \sum_{j=0}^\infty \delta_a V_j(1, \Delta_\varphi) t^{(j-m)/2}= 
     \sum_{j=0}^\infty    V_j(f_k, \Delta_\varphi) 
     t \frac{d}{dt} t^{(j-m)/2} .
\end{align*}
By equating the coefficients on two sides, we obtained:
\begin{prop}
    \label{prop:varprob-V_j-step1}
    Let $f_k = \delta_a(k^{1/2})k^{-1/2} + k^{-1/2}\delta_a(k^{1/2}) $ defined
    as before, for $j=0,1,2,\dots$, 
    \begin{align}
    \label{eq:varprob-V_j-step1}
        \delta_a V_j(1, \Delta_\varphi) = 
        \frac{j-m}{2}
        V_j(f_k , \Delta_\varphi) .
    \end{align}
\end{prop}
We only interested the second coefficient   in which $j=2$.
\begin{thm}
\label{thm:varprob-gradFeh-heat-op-side}
    Keep the notations. 
    We have the following variation formula for the Einstein-Hilbert action
    when the dimension $m>2$:
    \begin{align*}
        \delta_a F_{\mathrm{EH}}   
        = \varphi_0 (\delta_a(k) \grad_\mathrm{k} F_{\mathrm{EH}}),
    \end{align*}
in which
\begin{align}
    \begin{split}
       & \frac{2}{2-m} \grad_\mathrm{k} F_{\mathrm{EH}} \\
       = & \,\, 
    k^{j_m-1} \mathbf y^{-1/2} K_{\Delta_\varphi}(\mathbf y;m) (\nabla^2 k) \cdot
    g^{-1}\\
        &+
        k^{j_m -2} \mathbf y^{-1/2}_1 \mathbf y^{-1/2}_2
        H_{\Delta_\varphi}(\mathbf y_1, \mathbf y_2;m) (\nabla
        k \nabla k) \cdot
        g^{-1} \\
        = & \,\,  k^{j_m-1} K_{\Delta_k}(\mathbf y;m) (\nabla^2 k) \cdot
    g^{-1} +
    k^{j_m-2} H_{\Delta_k  }(\mathbf y_1, \mathbf y_2;m) (\nabla
        k \nabla k) \cdot
        g^{-1},
    \end{split}
    \label{eq:varprob-gradk-FEH}
\end{align}
    where $j_m = -m/2$. 
\end{thm}
\begin{proof}
 Notice that
 \begin{align*}
     f_k &= k^{-1/2} (1+ \mathbf y^{-1/2}) \delta_a(k^{1/2}) 
     = k^{-1} (1+ \mathbf y^{-1/2}) \frac{\mathbf y^{1/2}-1}{\mathbf y -1} 
( \delta_a (k) ) \\
     &= k^{-1} \mathbf y^{-1/2} (\delta_a (k) ).
 \end{align*}
 Hence
 \begin{align*}
     V_2(f_k, \Delta_\varphi) &= \varphi_0 ( f_k R_{\Delta_\varphi}) =
     \varphi_0 ( k^{-1} \mathbf y^{-1/2} (\delta_a (k)
     ) R_{\Delta_\varphi})\\
     &=\varphi_0( \delta_a (k)\mathbf y^{1/2} ( R_{\Delta_\varphi} k^{-1} )
         =\varphi_0 ( \delta_a (k) k^{-1} \mathbf y^{-1/2} (R_{\Delta_\varphi})
         )
     .
 \end{align*}
 That is 
 \begin{align*}
     \grad_\mathrm{k} F_\mathrm{EH} = k^{-1} \mathbf y^{-1/2}
     (R_{\Delta_\varphi}) = k^{-1} R_{\Delta_k}.
 \end{align*}
 To see \eqref{eq:varprob-gradk-FEH}, one just needs to substitute
 $R_{\Delta_\varphi}$ with its  explicit
 formula   stated in Theorem
 \ref{thm:varprob-scacur-for-Delta-varphi}.
\end{proof}

On the other side, we have computed (in Theorem \ref{thm:varprob-var-localT}
with $j = \tilde j_m = j_m-1
= -m/2-1$) the 
variation using the local expression
\begin{align*}
    \delta_a V_2(1,\Delta_\varphi) = \delta_a \varphi_0 \brac{
    k^{j_m -1} [T_{\Delta_\varphi} (\mathbf y) (\nabla k) ](\nabla k)
}, 
\end{align*}
Namely, 
\begin{align}
    \label{eq:varprob-gradFEH-from-localform}
    \begin{split}
     \grad_\mathrm{k} F_\mathrm{EH} &= 
     k^{j_m-1} [\sum_{l=1}^2\mathrm{I}^{(l)}(T)](\mathbf y;\tilde j_m)
     (\nabla^2 k)
     \cdot
    g^{-1} \\
    &+
    k^{j_m-2} [\sum_{l=1}^4\mathrm{II}^{(l)}(T)]  
    (\mathbf y_1, \mathbf y_2;\tilde j_m) (\nabla
        k \nabla k) \cdot
        g^{-1} .
   \end{split}
    \end{align}
The functional relations follows from \eqref{eq:varprob-gradFEH-from-localform}
 and \eqref{eq:varprob-gradk-FEH}:
    \begin{thm}
            \label{thm:varprob-thm-FunRelation-T-to-K-Ha}
        Let $T(u)\defeq T_{\Delta_\varphi}(u;m)$ be the spectral function that
        defines the Einstein-Hilbert action in Proposition
        $\ref{prop:varprob-Feh-closedform}$ and  $K_{\Delta_k}$,  $H_{\Delta_k}$
        are the spectral functions for the modular curvature $R_{\Delta_k}$ in
        Theorem $\ref{thm:psecal-R_Delta_k}$, 
        \begin{align}
            K_{\Delta_k} (u;m) = [\sum_{l=1}^2\mathrm{I}^{(l)}(T;\tilde
            j_m)](u),
            \,\,\,
            H_{\Delta_k} (u,v;m) = [\sum_{l=1}^4\mathrm{II}^{(l)}(T;\tilde j_m)]
            (u,v),
            \label{eq:varprob-thm-FunRelation-T-to-K-Ha}
        \end{align}
        where the righ hand sides are defined in Theorem
        $\ref{thm:varprob-var-localT}$ with $\tilde j_m = -m/2-1$.
    \end{thm}
    We remind the reader the explicit construction of the right hand sides in Eq.
    \eqref{eq:varprob-thm-FunRelation-T-to-K-Ha}. For the one-variable part:
    \begin{align}
        (\mathrm{I}^{(1)} + \mathrm{I}^{(2)} ) (u;j)=
        -T(u)-u^jT(u^{-1}) ,
        \label{eq:varprob-IT-explicit}
    \end{align}
    and the two-variable part:
    \begin{align}
\label{eq:varprob-IIT-explicit}
\begin{split}
     \mathrm{II}^{(1)}(T;j)(u,v) &= T(u) \frac{
            (uv)^j -1
        }{ uv -1}  ,\\
\mathrm{II}^{(2)}(T;j)(u,v) &=
        \frac{u^{j-1}
        \left(T\left(  u^{-1} \right)-T(v)\right)}{u^{-1}  -v}-\frac{u (u
        v)^{j-1} \left(T\left( v^{-1}  \right)-T(u)\right)}{  v^{-1}-u}
        ,\\
\mathrm{II}^{(3)}(T;j)(u,v) &=
        -\frac{\left(u^j-1\right) T(v)}{u-1}+\frac{v (T(u v)-T(v))}{u
        v-v}-\frac{T(u v)-T(u)}{u v-u}
        ,\\
\mathrm{II}^{(4)}(T;j)(u,v) &=
-\frac{\left(u^j-1\right) v^j T\left( v^{-1} \right)}{u-1}-\frac{(u
v)^j T\left( (uv)^{-1}  \right)-u^j T\left( u^{-1}  \right)}{u
v-u} \\
&+\frac{v \left((u v)^j T\left( (uv)^{-1} \right)-v^j
T\left(  v^{-1} \right)\right)}{u v-v} .
\end{split}
           \end{align}

\subsection{Verification of the functional relations}
In this section, we give explicit expression of the spectral functions 
$K_{\Delta_k}$, $H_{\Delta_k}$ and $T_{\Delta_k}$ as functions in $m$, the
dimension of the ambient manifold. The upshot is that the functional relations 
in Theorem \eqref{thm:varprob-thm-FunRelation-T-to-K-Ha} can be verified by
a computer algebra system once for all real value
$m$ with $m>2$. It is indeed a new theorem rather than a double validation, 
 because for general real-value $m$
with $m>2$, there are no geometric proofs so far.

We recall the main result in the last section,
\begin{align}
     \label{eq:b2cal-KDeltak-copied}
K_{\Delta_k} (y;m) &= \frac4m K_{3,1}(y;m) - K_{2,1}(y;m) \\
 &= 
    -\frac{1}{2} \Gamma \left(  m/2 +1\right)
     \pFq21 \left(  m/2 +1,1,3,1-s\right) 
     \\ &
     + \frac{2 \Gamma \left(m/2  +2\right)}{3m}  \pFq21\left(m/2  +2,1,4,1-s
     \right), \nonumber
 \end{align}
and
\begin{align}
    \begin{split}
     H_{\Delta_k}( s,t ;m) &= (\frac{4 }{m}+2) H_{2,1,1}\left(y_1,y_2;m\right)
     -\frac{4 y_1
     H_{2,2,1}\left(y_1,y_2 ;m\right)}{m}
     \\ & 
     -\frac{8
     H_{3,1,1}\left(y_1,y_2;m\right)}{m} \\
& =  \frac{1}{6m}    
 2 (m+2) \Gamma \left(m/2+2\right) F_1\left(m/2+2;1,1;4;1-s t,1-s\right)
 \\ & 
-\frac{1}{6m}\Gamma \left(m/2+3\right)  2 F_1\left(m/2+3;1,1;5;1-s
t,1-s\right)\\ & 
-\frac{1}{6m}\Gamma \left(m/2+3\right) s F_1\left(m/2+3;1,2;5;1-s t,1-s\right)
.
\label{eq:b2cal-HDelta-hgfun-copied}
     \end{split}
   \end{align}
The evaluation of $F_1$ family appeared in
\eqref{eq:b2cal-HDelta-hgfun-copied} can be reduced to $\pFq21$ according to
Proposition \ref{prop:hgeofun-cor-F1-dividedidff}. On the other hand, 
   \textsf{Mathematica} is able to provide symblic evaluation in the parameter
   $a$ for $\pFq21(a,b;c,z)$ when other two parameters $b,c$ are given
   numerically.  

   \begin{prop}
       For any $m=\dim M >2$, the modular curvature $R_{\Delta_k}$ (see Theorem
       $\ref{thm:psecal-R_Delta_k}$) has the following explicit modular
       functions    
\begin{align}
    \begin{split}
        K_{\Delta_k}(s;m) &= \frac{
  -8 s^{-\frac{m}{2}} \left((m (s-1)-4 s) s^{m/2}+s (m (s-1)+4)\right) \Gamma
  \left(\frac{m}{2}+2\right)      
        }{
        (m-2) m^2 (m+2) (s-1)^3
    } ,
    \end{split}
\label{eq:checkrels-explict-KDelta}
\end{align}
and
\begin{align}
\label{eq:checkrels-explict-HDelta}
    \begin{split}
        H_{\Delta_k}(s,t;m) &= \frac{2}{m}  (s-1)^{-2} (t-1)^{-2} (s t-1)^{-3}
    \Gamma(m/2+1) \\ &      \left[ 
   2s^{-m/2} (s t-1)^3+
   2 (t-1)^2 \left(\frac{1}{2} m (s-1) (s t-1)+s (1-2 s) t+1\right)
     \right.
    \\
    & \left. 
        -2 (s-1)^2 t (s t)^{-\frac{m}{2}} \left(\frac{1}{2} m (t-1) (s t-1)+s
        t^2+t-2\right)  
    \right].
    \end{split}
      \end{align}
      Moreover, the symbolic evaluations still hold when $m \in (2,\infty)$ is
      a real parameter.
\end{prop}
Now, let us move on to  $T_{\Delta_k}$  in
\eqref{eq:varprob-defn-TDleta}. First of all,
\begin{align}
    K_{\Delta_k}(1) = 
    \frac{2 \Gamma \left(\frac{m}{2}+2\right)}{3 m}-\frac{\Gamma
    \left(\frac{m}{2}+1\right)}{2}.
    \label{eq:checkrel-defn-KDelta-at-0}
\end{align}
To compute $H_{\Delta_k}(y,y^{-1})$, one can, of course, compute the limit of the
right hand side of \eqref{eq:checkrels-explict-HDelta} as $t \rightarrow
s^{-1}$. Nevertheless, a better way is to use the reduction formula
\eqref{eq:hgfun-F1toGF21-1} of $F_1$, which implies that 
\begin{align}
    \label{eq:checkrel-Habc-reduced-to-Kab}
    H_{a,b,c} (u,u^{-1};m) = K_{a+c,b}(u;m).
\end{align}
\begin{prop}
    In terms of hypergeometric functions,
\begin{align}
    \label{eq:check-TDelta-using-GF21}
    T_{\Delta_\varphi}(y) &=  -K_{\Delta_\varphi}(1) \frac{s^{-m/2} -1 }{s-1}
    + \frac13 (1+ \frac2m) \Gamma \left(\frac{m}{2}+2\right) \pFq21
    ( \frac{m}{2}+2,1;4;1-s )  \\
    &- \frac{1}{3m} \Gamma \left(\frac{m}{2}+3\right) \pFq21
    ( \frac{m}{2}+3,1;5;1-s ) \nonumber \\
&- \frac{1}{6m} \Gamma \left(\frac{m}{2}+3\right)
 s \sbrac{\pFq21 ( \frac{m}{2}+3,2;5;1-s)}  
 \nonumber.
\end{align}
For $m \in (2,\infty)$, the right hand side above has the following evaluation:
      \begin{align}
          \begin{split}
              T_{\Delta_k}(s;m) &= \Gamma(a-1)
              6 a^{-1} (s-1)^{-4} s^{-a} \left[ 
-3 a^2 (-1 + s)^2 (-1 + s + s^a (1 + s))
              \right. \\ & \left.
                  +2 a ((-1 + s)^3 + (-1 + s) s^a (-2 + s (7 + s)))
              \right. \\ & \left.
                  a^3 (s-1)^3 \left(s^a+1\right)-12 s^2 \left(s^a-1\right)
              \right] , \,\,\,
              \text{ where $a =m/2$.}
          \end{split}
    \label{eq:checkrels-explict-TDelta}
      \end{align}
\end{prop}

\begin{thm}
    The explicit formulas for $K_{\Delta_k}(u;m)$, $H_{\Delta_k}(u,v;m)$ and
    $T_{\Delta_k}(u;m)$ defined in \eqref{eq:checkrels-explict-KDelta},
    \eqref{eq:checkrels-explict-HDelta} and \eqref{eq:checkrels-explict-TDelta}
    satisfy the functional relations stated in Theorem
    $ \ref{thm:varprob-thm-FunRelation-T-to-K-Ha}$,   even for real parameter
    $m$ as long as $m \in (2,\infty)$. 
\end{thm}



\section{ Symbolic computation 
}
\label{sec:symbloic-computation}
The explicit computation for the $b_2$ term in the resolvent approximation (see
\eqref{eq:intro-b_j-Fsum})
has been performed several times in
different settings and via different methods: 
\cite{MR3540454,MR3194491,MR3359018,MR3148618,MR2956317,MR2907006,
Liu:2015aa,LIU2017138}.
In this section, we shall only outline some key steps. As pointed out before,
It is sufficient to
carried out the computation for noncommutative $m$-tori with the flat
background metric. 
Instead of  Connes's
pseudo-differential calculus \cite{connes1980c}, we use the deformed Widom's
calculus with respect to a flat connection so that the full expression of the
$b_2$
term can be recorded in a fews lines. Notations used in this section are highly
compatible with those in \cite{Liu:2015aa}.

Let us quickly review the algorithm of constructing resolvent approximation for
differential operators via pseudo-differential calculus.
We assume  the pseudo-differential operators considered in the sections are
scalar operators, acting on smooth functions.
Symbols of pseudo-differential operators form a subalgebra inside  
smooth functions on the cotangent bundle which admits a filtration. The
associated graded algebra is called the algebra of complete symbols.
Let $P$ and $Q$ be two pseudo-differential operators
with symbol $p$ and $q$ respectively. Then the symbol of their composition has
a formal expansion
\begin{align*}
    \sigma(PQ) =  p \star q \backsim
    \sum_{j=0}^\infty a_j(p,q),
\end{align*}
where each $a_j(\cdot,\cdot)$ is a bi-differential operator such that $a_j(p,q)$
reduce the total degree by $j$. 

Consider a second order differential operator $P$ with symbol $\sigma(P) =p_2+p_1 + p_0$,
where $p_j$ homogeneous of order $j$ with $j=0,1,2$. In most of 
pseudo-differential calculi, $a_0$ has no differential, here we assume that
$a_0(p,q) = p q$.  With the initial value $b_0 = (p_2 - \lambda)^{-1}$, one can
recursively construct $b_j$: 
\begin{align}
    \label{eq:b2cal-b1term-gen}
 b_1 &= [ \mathit{a}_0\left(b_0,p_1\right)+\mathit{a}_1\left(b_0,p_2\right)](-b_0)
\\
 b_2 &= 
 [\mathit{a}_0\left(b_0,p_0\right)+\mathit{a}_0\left(b_1,p_1\right)+\mathit{a}_1\left(b_0,p_1\right)+\mathit{a}_1\left(b_1,p_2\right)+\mathit{a}_2\left(b_0,p_2\right)] (-b_0).
\label{eq:b2cal-b2term-gen}
\end{align}
The construction can be continued while the complexity of the right hand sides increases dramatically.    

We now specialize on noncommutative tori (of arbitrary dimension $\ge 2$) from
deformation point of view. Let $M = \T^m$, a $m$-dimensional
torus  with the flat Euclidean  metric, and let $\nabla$ be the
Levi-Civita connection.  For two symbols $p = p(x,\xi)$ and $q =q(x,\xi) $, with $(x,\xi) \in T^*\T^m
\cong \T^m \times \R^m$, we have
\begin{align}
    a_j(p,q) = \frac{(-i)^j}{j!} (D^j p) \cdot (\nabla^j q),   
    \,\, j=0,1,2,\dots,
    \label{eq:b2cal-a_j-flatconnection}
\end{align}
where $D = D_\xi$ is the vertical differential so that $ (D^j p)$ is
a contravariant $j$-tensor fields. Deformed contracting (what the $\cdot$
stands for) with a covariant
$j$-tensor field $\nabla^j q$ gives rise to a scalar tensor field which is also
a symbol of order $\deg p+\deg q -j$. For more explanations about the notations, see \cite{Liu:2015aa,LIU2017138}.

Consider the perturbed Laplacian $\Delta_k := k \Delta$,  with the heat operator
as a   contour integral,
\begin{align*}
    e^{-t\Delta_k} = \int_C e^{-t\lambda} (\Delta_k -\lambda)^{-1} d\lambda.
\end{align*}
If we ignore the zero  spectrum of $\Delta_j$, the contour $C$ can be chosen to be   the imaginary axis from $-i\infty$ to $i\infty$. 
%
%

Any finite sum $\sum_{j=0}^N b_j$ will give an approximation of the resolvent
$(\Delta_k - \lambda)^{-1}$ which leads to the asymptotic expansion:   
\begin{align*}
    \Tr(a e^{-t\lambda} )\backsim_{t\searrow 0} \sum_{j=0}^\infty t^{(j-m)/2}
    V_j(a,\Delta_k) = \varphi_0 \brac{
a \mathcal R_j
    }, \,\,\, \forall a \in C^\infty(\T^m_\theta),
\end{align*}
where $\mathcal R_j$ is the functional density which can be explicitly computed by $b_j$: 
\begin{align}
    \label{eq:psecal-mathcalR_j}
    \mathcal R_j = \int_{T^*_xM} \frac{1}{2\pi i} \int_C e^{-\lambda} 
    b_j(\xi,\lambda) d\lambda d\xi. 
\end{align}

The symbol of $\sigma(\Delta_k) = \sigma(k \Delta) = p_2 + p_1 +p_0$ is
a degree two polynomial in $\xi$: 
\begin{align*}
    p_2 = k \abs\xi^2, \,\,\, p_1 =p_0 =0.
\end{align*}

As a function on the cotangent bundle, one can compute the vertical $D$ and
horizontal $\nabla$ differential of $p_2$ as below:
\begin{align}
    \begin{split} 
    (D p_2)_j = 2 \xi_j, \,\, & (D^2 p_2)_{jl} = 2 \mathbf 1_{j l}, 
    \\
    (\nabla p_2)_j = (\nabla k)_j \abs\xi^2, \,\,\,& (\nabla^2 p_2)_{jl} = (\nabla^2
    k)_{jl} \abs\xi^2,
    \end{split}
    \label{eq:b2cal-vetandhordiffofp2}
\end{align}
where $\mathbf 1_{jl}$ stands for the Kronecker-delta symbol. 
By substituting   \eqref{eq:b2cal-a_j-flatconnection} and the derivatives of the
symbols \eqref{eq:b2cal-vetandhordiffofp2} into general formula 
\eqref{eq:b2cal-b2term-gen}, we obtain the expanded $b_2$ term as a function on
the cotangent bundle,
\begin{align}
\begin{split}
    b_2(\xi) = & \,\,  4 r^2 \xi _j \xi _l b_0^3.k^2. (\nabla^2 k)_{l,j}.b_0-r^2. \mathbf 1_{j l}. b_0^2.k. (\nabla^2 k)_{l,j}.b_0 \\
 & \,\,  4 r^2 \xi _j \xi _l b_0^2.k. (\nabla k)_l.b_0. (\nabla k)_j.b_0 
 -4 r^4 \xi _j \xi _l b_0^2.k. (\nabla k)_l.b_0^2.k. (\nabla k)_j.b_0,
\\
&\,\,
2 r^4 .\mathbf 1_{j l} .b_0^2.k.(\nabla k)_l.b_0.\nabla (k)_j.b_0-8 r^4 \xi _j \xi _l b_0^3.k^2.(\nabla k)_l.b_0.(\nabla k)_j.b_0,
\end{split}
\label{eq:b2cal-b2termwithxi}
 \end{align}
 where the summation is taken automatically for repeated indices from $1$ to $m
 =\dim \T^m$. 
 Let $r = \abs\xi^2$, where the length is associated to the flat Riemannian metric.
 We will compute the integration over the fiber $T_x^* \T^m$ using spherical
 coordinates:
 \begin{align*}
     \int_{T_x^* \T^m} b_2(\xi) d\xi = \int_0^\infty b_2(r) r^{m-1}dr,
 \end{align*}
 with  
 \begin{align*}
     b_2(r) = \int_{\abs\xi^2 = 1} b_2(\xi) d\mathbf{s},
 \end{align*}
 where $d\mathbf s$  the  standard volume form for the unit sphere in
 $\R^{m-1}$.

 \begin{lem}
     Let $d\mathbf s$ be the  standard volume form for the unit sphere in
     $\R^{m-1}$, we have
     \begin{align*}
         \mathrm{Vol}(S^{m-1}) &= 
         \int_{S^{m-1}}    d\mathbf{s} = \frac{2 \pi^{m/2}}{\Gamma(m/2)},
\\
         \int_{S^{m-1}} \xi_j \xi_l   d\mathbf{s} &=
         \frac{\pi^{m/2}}{\Gamma(1+m/2)} \mathbf{1}_{jl}
         = \mathrm{Vol}(S^{m-1}) \frac1m \mathbf 1_{jl}.
     \end{align*}
 \end{lem}
 \begin{proof}
     Elementary calculus, left to the reader.
 \end{proof}
Upto an overall constant factor $\mathrm{Vol}(S^{m-1})$, $b_2(r)
 = \int_{S^{m-1}} b_2 (\xi)
 d \mathbf s$ equals 
 \begin{align}
 \begin{split}
 b_2(r)  = &\,\,
 -r^2 \mathbf 1_{j l}. b_0^2.k.(\nabla^2 k)_{l,j}.b_0 + \frac{4 r^4 \mathbf 1_{j
 l} b_0^3.k^2.(\nabla^2 k)_{l,j}.b_0}{m} \\
 &\,\,
 +2 r^4 \mathbf 1_{j l}. b_0^2.k.(\nabla k)_l.b_0.(\nabla k)_j.b_0 +
 \frac{4 r^4 \mathbf 1_{j l}. b_0^2.k.(\nabla k)_l.b_0.(\nabla k)_j.b_0}{m} \\
 &\,\, 
 -\frac{8 r^6 \mathbf 1_{j l}. b_0^3.k^2.(\nabla k)_l.b_0.(\nabla k)_j.b_0}{m}
 -\frac{4 r^6 \mathbf 1_{j l}. b_0^2.k.(\nabla k)_l.b_0^2.k.(\nabla
 k)_j.b_0}{m}   . 
 \end{split}
\label{eq:b2cal-b_2inr}
 \end{align}
 The summation over $i,j$ stands for contraction between contravariant and
 covaraint tensors. 
  To be precise: 
  \begin{align}
      \label{eq:psecal-contractiontoLaplacian} 
  (\nabla^2 k)\cdot g^{-1}  & \defeq \sum \mathbf 1_{l j} (\nabla^2 k)_{l,j}=
  \Tr \Hess (k)
  = -\Delta k, \\
  (\nabla k \nabla k) \cdot g^{-1} &\defeq \sum \mathbf 1_{lj} (\nabla k)_l
  (\nabla k)_j
  = \abrac{\nabla k , \nabla k}_g,
  \label{eq:psecal-contractiontonormofdk}
  \end{align}
  where $g^{-1}$ stands for the metric on the cotangent bundle. 
 We now give some examples on how to apply integration lemma developed at
 the beginning of the paper. Recall $b_0 = (k r^2 -\lambda)^{-1}$. We first
 move powers of $k$  in front, for instance,
 \begin{align*}
     r^2 \mathbf 1_{j l}. b_0^2.k.(\nabla^2 k)_{l,j}.b_0
     &= k \brac{ 
 r^2 \mathbf 1_{j l}. b_0^2.(\nabla^2 k)_{l,j}.b_0 } \\
 r^6 \mathbf 1_{j l}. b_0^2.k.(\nabla k)_l.b_0^2.k.(\nabla k)_j.b_0
 &= k^2 \mathbf y_1 \brac{
 \mathbf 1_{j l}. b_0^2.(\nabla k)_l.b_0^2.(\nabla k)_j.b_0
 }
 \end{align*}
 where in the second line, $\mathbf y_1$ is the conjugation operator acting on
 the  factor  $(\nabla k)_l$, which allows us to move the $k$ between
 $(\nabla k)_l$ and $(\nabla k)_j$ in front. Then we apply Proposition
 \ref{prop:hgeofun-one-var-family}:
 \begin{align*}
    & \,\, \int_0^\infty \frac{1}{2 \pi i} \int_C e^{-\lambda}
     r^2 \mathbf 1_{j l}. b_0^2.k.(\nabla^2 k)_{l,j}.b_0 d\lambda
     (r^{m-1} dr) 
     \\ = & \,\,
     k^{-(m/2+1)}
     K_{2,1}(\mathbf y;m)(\nabla^2 k) \cdot g^{-1}, 
 \end{align*}
 and Proposition  \ref{prop:hgeofun-two-var-family}:
 \begin{align*}
     & \,\, \int_0^\infty \frac{1}{2 \pi i} \int_C e^{-\lambda}
\mathbf 1_{j l}. b_0^2.(\nabla k)_l.b_0^2.(\nabla k)_j.b_0
d\lambda (r^{m-1} dr)  
\\ = & \,\,
k^{-(m/2+3)}
H_{2,2,1}(\mathbf y_1, \mathbf y_2 ;m)(\nabla k \nabla k) \cdot g^{-1}. 
 \end{align*}

 The one variable spectral function is obtained by integrating the
 first two terms in \eqref{eq:b2cal-b_2inr}:
 \begin{align}
     K_{\Delta_k} (y;m) = \frac4m K_{3,1}(y;m) - K_{2,1}(y;m)
     \label{eq:b2cal-KDeltak}
 \end{align}
 and for the two variable function, it comes from the last four terms in
 \eqref{eq:b2cal-b_2inr}:
 \begin{align}
\label{eq:b2cal-HDeltak}
\begin{split}
      H_{\Delta_k} (y_1 , y_2 ;m)  &=
 (\frac{4 }{m}+2) H_{2,1,1}\left(y_1,y_2;m\right)  -\frac{4 y_1
     H_{2,2,1}\left(y_1,y_2 ;m\right)}{m}
     \\ &  
     -\frac{8
     H_{3,1,1}\left(y_1,y_2;m\right)}{m}.
\end{split}
    \end{align}
    Let us summerize the computation of the whole section as a theorem.
\begin{thm}
    \label{thm:psecal-R_Delta_k}
    For the perturbed Laplacian $\Delta_k = k \Delta$, a closed form of the
    functional $V_2(\cdot,\Delta_k)$ is given by
    \begin{align*}
        V_2(a, \Delta_k) = \varphi_0( a R_{\Delta_k}), \,\,\,
        \forall a \in C^\infty(\T^m_\theta),
    \end{align*}
    with $R_{\Delta_k} \in C^\infty(\T^m_\theta)$. Upto an overall constant
    $\op{Vol}(S^{m-1})/2$, 
    \begin{align}
    \label{eq:psecal-R_Delta_k}
        R_{\Delta_k} = k^{j_m} K_{\Delta_k} (\mathbf y ;m) (\nabla^2 k)
        \cdot g^{-1} + k^{j_m-1}
        H_{\Delta_k}(\mathbf y_1, \mathbf y_2 ;m) (\nabla k \nabla k) \cdot
        g^{-1}, 
    \end{align}
    where $j_m = -m/2$, $\mathbf y$ and $\mathbf y_l$ with $l=1,2$ are the
    modular operators, see 
    \eqref{eq:hgeofun-y_j(k)}. Contractions with the metric $g^{-1}$ are
    explained in \eqref{eq:psecal-contractiontoLaplacian} and
    \eqref{eq:psecal-contractiontonormofdk}. Functions
    $K_{\Delta_k}$ and $H_{\Delta_k}$ are
    defined in \eqref{eq:b2cal-KDeltak} and \eqref{eq:b2cal-HDeltak}
    respectively.
\end{thm}

In terms of hypergeometric functions, 
\begin{align}
    \begin{split}
        K_{\Delta_k}( s ;m) = & \,\,
    -\frac{1}{2} \Gamma \left(  m/2 +1\right)
     \pFq21 \left(  m/2 +1,1,3,1-s\right) 
     \\ &\,\,
     + \frac{2 \Gamma \left(m/2  +2\right)}{3m}  \pFq21\left(m/2  +2,1,4,1-s
     \right),
     \end{split}
\label{eq:b2cal-KDelta-hgfun}
    \end{align}
    and
\begin{align}
    \begin{split}
     H_{\Delta_k}( s,t ;m) = 
& \,\, \frac{1}{6m}    
 2 (m+2) \Gamma \left(m/2+2\right) F_1\left(m/2+2;1,1;4;1-s t,1-s\right)
 \\ & \,\,
-\frac{1}{6m}\Gamma \left(m/2+3\right)  2 F_1\left(m/2+3;1,1;5;1-s
t,1-s\right)\\ & \,\,
-\frac{1}{6m}\Gamma \left(m/2+3\right) s F_1\left(m/2+3;1,2;5;1-s t,1-s\right).
     \end{split}
\label{eq:b2cal-HDelta-hgfun}
   \end{align}

\appendix                                     

\section{Hypergeometric functions}

Hypergeometric functions had been studied intensively in the nineteenth
century. The pioneers include Gauss (1813), Ernst Kummer (1836) and Riemann
(1857), etc. The two variable extension of the hypergeometric functions has
four different types known as Appell's $F_1$ to $F_4$ functions. 
In this appendix, we only collect  some basic knowledge of $\pFq21$ and $F_1$
functions that are related to our exploration of modular curvature.
Most of the
identities  quoted in this section can be found in \cite{MR0058756},
\cite{MR2723248}, \cite{appell1925fonctions} and \cite{appell1926fonctions}. 

\subsection{Gauss Hypergeometric functions}
For $\abs z<1$, the (Gauss) hypergeometric function $\pFq21(a,b;c;z)$ is
represented by the hypergeometric series
\begin{align}
    \label{eq:hgeofun-defn-GF21-series}
    \pFq21(a,b;c;z) = \sum_{n=0}^\infty \frac{
    (a)_n (b)_n
    }{
        (c)_n
    } \frac{ z^n }{ n!} ,
\end{align}
where the coefficients are given by  Pochhammer symbols: 
\begin{align}
    (q)_n = \frac{ \Gamma(q+n)}{ \Gamma(q)}.
    \label{eq:hggeofun-defn-Pochh-sym}
\end{align}
What we need in the paper is the   following Euler type integral
representation:
\begin{align}
 \pFq21(a,b;c; z)
 = \frac{\Gamma(c) }{ \Gamma(b) \Gamma(c-b)}     
 \int_0^1 (1-t)^{c-b-1} t^{b-1} (1-z t)^{-a} dt.
    \label{eq:hgeo-defn-pFq21}   
\end{align}
It is a solution of Euler's hypergeometric differential equation
\begin{align}
    z(1-z) \frac{d^2 w}{dz^2} + \brac{
        c-(a+b+1)z
    } \frac{dw}{dz} - ab w =0 
    \label{eq:hgeofun-HGODE}
\end{align}
For $F \defeq \pFq21(a,b;c;z) $, there are six associated contiguous functions
obtained by applying $\pm1$ on only one of the parameters $a,b$ and $c$. We
abbreviate them as 
$F(a+)$, $F(b+)$, $F(c-)$, etc.  Gauss showed that $F$ can be written as
a linear combination of any two of its contiguous functions, which leads to
$15$ ($6$ choose $2$) relations. They can be derived from the differential
relations among the
family 
\begin{align}
    \label{eq:hgfun-diff-relations-GF21-abc+1}
    \frac{d}{dz}  ( \pFq21(a,b;c;z) ) = \frac{a b}{c} \pFq21(a+1,b+1;c+1;z),
\end{align}
also
\begin{align}
    \label{eq:hgfun-diff-relations-GF21-a+1}
    F(a+) &= F + \frac1a z\frac{d}{dz} F, \\
    \label{eq:hgfun-diff-relations-GF21-b+1}
    F(b+) &= F + \frac1b z\frac{d}{dz} F, \\
    F(c-) &= F + \frac1c z\frac{d}{dz} F. 
    \label{eq:hgfun-diff-relations-GF21-c-1}
\end{align}
Combine with the second order ODE \eqref{eq:hgeofun-HGODE}, we have 
\begin{prop}
    \label{prop:hgeofun-cont-relations}
 One can read all $15$ relations among the contiguous functions  by equating
 any two lines of the right hand
 side of \normalfont{Eq.} \eqref{eq:hgeofun-cont-relations}: 
\begin{align}
    \begin{split}
        z \frac{dF}{dz} &= a \frac{ab}{c} F(a+,b+,c+) \\
        &=a(F(a+) -F) \\
        &=b(F(b+) -F) \\
        &=(c-1) (F(c-) -F) \\
        &= \frac{
            (c-a) F(a-) + (a-c+bz)F
        }{ 1-z}  \\
&= \frac{
            (c-b) F(b-) + (b-c+az)F
        }{ 1-z} \\
        &= z \frac{
            (c-a)(c-b) F(c+) + c(a+b-c)F
        }{ c(1-z)}.
    \end{split}
    \label{eq:hgeofun-cont-relations}
\end{align}
\end{prop}
There are other type of symmetries among the hypergeometric family. For
example, under fractional linear transformation
\begin{align}
    \begin{split}
        \pFq21(a,b;c;z)& = (1-z)^{-b} \pFq21(c-a,b;c;\frac{z}{z-1}) \\
        \pFq21(a,b;c;z) &= (1-z)^{-a} \pFq21(a,c-b;c;\frac{z}{z-1}).
    \label{eq:hgeofun-Pfaff-1}
    \end{split}
\end{align}
They are known as Pfaff transformations. Then the Euler transformation 
\begin{align}
    \label{eq:hgeofun-Euler}
    \pFq21(a,b;c;z)&= (1-z)^{c-a-b} \pFq21(c-a,c-b;c;z)
\end{align}
follows quickly.

\begin{prop}
    For $p\in \Z_{>0}$, 
    \begin{align}
        \pFq21(a,1;c;z) = \frac{(1-c)_p}{(a-c+1)_p} \pFq21(a,1;c-p;z)
        + \frac1z \sum_{k=1}^p \frac{(1-c)_k}{(a-c+1)_k} 
        \brac{
            \frac{z-1}{z}
        }^{k-1}.
        \label{eq:hgfun-EVGF21a1c}
    \end{align}
    Since $ \pFq21(a,1;1;z) = (1-z)^{-a}$, when $a$ and $c$ belong to the natural domain of the right hand side of
    \eqref{eq:hgfun-EVGF21a1c}, it provides a symbol evaluation for
    $\pFq21(a,1;c;z)$.
\end{prop}

\subsection{Appell Hypergeometric functions}
The hypergeometric series \eqref{eq:hgeofun-defn-GF21-series} has a variety of
generalizations to multi-variable cases. For the two-variable case, Appell
introduced four types of series $F_1$ to $F_4$. So far, $F_1$ is directly
related to the modular curvature functions. To deal with some symbolic
computation, we need  $F_2$ as a bridge. 
\begin{align}
    F_1(a;b,b';c;x;y) &= \sum_{m=0}^\infty \sum_{n=0}^\infty
    \frac{
        (a)_{m+n} (b)_m (b')_n
    }{m! n! (c)_{m+n} } x^m y^n
    \label{eq:hgfun-defn-series-APF1} \\
    F_2(a;b,b';c,c';x,y) &= \sum_{m=0}^\infty \sum_{n=0}^\infty
\frac{
        (a)_{m+n} (b)_m (b')_n
    }{m! n! (c)_{m+n}  (c')_{n+m} } x^m y^n,
\label{eq:hgfun-defn-series-APF2}
\end{align}
where Pochhammer symbols \eqref{eq:hggeofun-defn-Pochh-sym} is used in the
coefficients.  
They have a double integral representation:
\begin{align}
    \begin{split}
    &\,\, \frac{\Gamma(b) \Gamma(b') \Gamma(c-b-b')}{\Gamma(c)} 
    F_1(a;b,b',c;x,y) \\
    =& \,\,
    \int_0^1 \int_0^{1-t} u^{b-1} v^{b'-1} (1-u-v)^{c-b-b'-1}
    (1-x u - y v)^{-a} dudv  .  
\end{split}
    \label{eq:hygeofun-APF1-defn} 
    \\
    \begin{split}
        &\,\, \frac{\Gamma(b) \Gamma(b') \Gamma(c-b)
            \Gamma(c'-b')
        }{\Gamma(c) \Gamma(c')
        } 
        F_2(a;b,b';c,c';x,y) \\
    =& \,\,
    \int_0^1 \int_0^{1} u^{b-1} v^{b'-1} (1-u)^{c-b-1}
(1-v)^{c'-b'-1}
    (1-x u - y v)^{-a} dudv  .  
    \end{split}
    \label{eq:hygeofun-APF2-defn}
\end{align}

Parallel to the differential system  for  Gaussian hypergeometric functions, we
have a similar relations for rising the parameters via differential operators: $x\partial_x$ and $y\partial_y$.
Denote by $F_1 \defeq F_1(a;b,b',c;x,y)$ and $F_1(a+)
\defeq F_1 ( a+1;b,b',c;x,y)$,  same pattern applies to $F_1(b+)$, $F_1(b'+)$ and
$F_1(c+)$, then
\begin{align}
    \partial_x F_1 = F_1(a+,b+,c+), \,\,
    \partial_y F_1 = F_1(a+,b'+,c+), 
    \label{eq:hgeofun-f1abc+-diffsys} 
\end{align}
also
\begin{align}
    F_1(a+) &=   a^{-1}(a+x\partial_x + y\partial_y)F_1 
    \label{eq:hgeofun-f1a+-diffsys} \\
F_1(b+) &=  b^{-1}( b+x\partial_x )F_1 
\label{eq:hgeofun-f1b+-diffsys} \\
    F_1(b'+) &=  b'^{-1}( b'+y\partial_y )F_1 
    \label{eq:hgeofun-f1b'+-diffsys} \\
F_1 &=   c^{-1}(c+ x \partial_x + y\partial_y )F_1(c+) 
\label{eq:hgeofun-f1c+-diffsys}.
\end{align}
For $F_1$ itself, it is a solution of the PDE system:
\begin{align}
    \sbrac{
    x(1-x) \partial_x^2 + y(1-x) \partial_x \partial_y
    + [c-(a+b+1)] \partial_x - by \partial_y -ab
}F_1 &=0
    \label{eq:hgeofun-} \\
 \sbrac{
    y(1-y) \partial_y^2 + x(1-y) \partial_x \partial_y
    + [c-(a+b'+1)] \partial_x - b'x \partial_y -ab'
}F_1 &=0
\end{align}

The $F_1$ family reduces to the hypergeometric functions in the situations:
\begin{align}
    \label{eq:hgfun-F1toGF21-2}
    \begin{split}
    F_1(a;b,b';c;0,y) &= \pFq21(a,b';c;y), \\
    F_1(a;b,b';c;x,0) &= \pFq21(a,b;c;x).    
    \end{split}
    \end{align}
In addition,
\begin{align}
    F_1(a;b,b';c;x,x) &= (1-x)^{c- a-b-b'} 
    \pFq21(c-a;c-b-b';c;x) \nonumber\\
    &= \pFq21(a,b;b+b';x),  \nonumber\\
    F_1(a;b,b';b+b';x,y) &= (1-y)^{-a} \pFq21(a,b;b+b';\frac{x-y}{1-y}).
    \label{eq:hgfun-F1toGF21-1}
\end{align}

The $F_1$ family can be computed via the reduction formula, cf. \cite[Sec.
16.16]{MR2723248} or \cite[Sec. 5.10, 5,11]{MR0058756},
\begin{align}
\begin{split}
     F_1(a,b,b',c;x,y) &= (x/y)^{b'} F_2 (b+b';a,b';c,b+b',x,1-x/y) \\
     &= (y/x)^{b} F_2(b+b';a,b;c,b+b',y,1-y/x)       .
    \end{split}
    \label{eq:hgeofun-APF2toAPF1}
\end{align}

\begin{prop}[\cite{MR2107356}, Theorem 2]
    For $a \in \mathbb{C}$, $b \in \mathbb{C}\setminus \Z_{\le 0}$,  
    $p,q \in \Z_{\ge 0}$,  $p<q$ 
    and $\abs x + \abs y <1$,
    \begin{align}
       \begin{split}
        &\,\, F_2(q+1,a,p+1;b,p+2;x,y) \\
        =&\,\,
        -\frac{p!}{q(1-q)_p} \frac{p+1}{y^{p+1}} 
        \pFq21 (a,q-p;b;x)  \\ 
        + & \frac{p+1}{y^{p+1}} \sum_0^p \frac{(-1)^{k}}{(q-k) (1-y)^{q-k}}     
        \binom pk \sum_0^{p-k} (-x)^{m} \binom{p-k}{m} \frac{ (a)_m}{(b)_m} \\
        \cdot & \pFq21 \brac{a+m,q-k;b+m; \frac{x}{1-y} } .\\
        \end{split}
        \label{eq:hpgeofun-F2toGF21-thm}
    \end{align}
\label{prop:hpgeofun-F2toGF21-thm}
\end{prop}


\section*{Acknowledgement}
The author would like to thank MPIM, Bonn for providing marvelous working
environment. 
\bibliographystyle{halpha}

\bibliography{mylib}

\end{document}